\documentclass[11pt]{amsart}
\usepackage{geometry, stmaryrd, jswmacros,xypic}
\usepackage{graphicx}
\usepackage{amssymb}
\usepackage{hyperref}
\DeclareMathOperator{\height}{height}
\DeclareMathOperator{\Fr}{Fr}
\DeclareMathOperator{\Gr}{Gr}
\newcommand{\kk}{k}

\title{Explicit Non-abelian Lubin-Tate theory for $\GL_2$}
\author{Jared Weinstein}

\begin{document}
\maketitle

\begin{abstract}  Let $F$ be a non-Archimedean local field with residue field $k$ of odd characteristic, and let $B/F$ be the division algebra of rank 4.  We explicitly construct a stable curve $\mathfrak{X}$ over the algebraic closure of $k$ admitting an action of $\GL_2(F)\times B^\times \times W_F$ which realizes the Jacquet-Langlands correspondence and the local Langlands correspondence in its cohomology.
\end{abstract}

\section{Introduction}

Let $F$ be a non-Archimedean field;  {\em i.e.} a finite extension either of $\Q_p$ or the field of Laurent series over a finite field, and let $\OO_F$ be its ring of integers.  Let $n\geq 1$ be an integer, and let $B/F$ be the central division algebra of invariant $1/n$.  There are well-known correspondences between representations of the groups $\GL_n(F)$, $B^\times$, and $W_F$, the Weil group of $F$.  These are the Jacquet-Langlands correspondence $\pi\mapsto\JL(\pi)$ (between $\GL_n$ and $B^\times$) and the local Langlands correspondence $\pi\mapsto\sigma(\pi)$

Loosely speaking, {\em non-Abelian Lubin-Tate theory} refers to the construction of a geometric object $\mathfrak{X}$ which realizes these correspondences simultaneously in its cohomology.  That is, one finds an action of the triple product group $\GL_n(F)\times B^\times \times W_F$ on the Euler characteristic of $\mathfrak{X}$ (computed with respect to an appropriate cohomology theory), which decomposes as a formal sum of representations of the form $\pi\otimes\pi'\otimes\sigma$, where $\pi$ is a representation of $\GL_n(F)$, $\pi'$ is the representation of $B^\times$ which corresponds to $\pi$ under the Jacquet-Langlands correspondence, and $\sigma$ is a representation of $W_F$ which corresponds to $\pi$ under the (suitably normalized) Langlands correspondence.

The case of $n=1$ is classical Lubin-Tate theory~\cite{LubinTate}, in which the isomorphism $\GL_1(F)=F^\times\isom W_F^{\text{ab}}$ of local class field theory is established through the study of division points of a one-dimensional commutative formal $\OO_F$-module of height 1.  For higher $n$, Carayol~\cite{Carayol:nonabelianLT} offered two approaches to the construction of the space $\mathfrak{X}$.  In the {\em vanishing cycle} setting, the role of  $\mathfrak{X}$ is played by the rigid generic fiber of the projective system $M^n_{\text{LT},F}$ of formal schemes representing the functor of deformations of a fixed formal $\OO_F$-module of height $n$ with Drinfeld level structures of all degrees.   In the {\em rigid} setting, the space $\mathfrak{X}$ is a projective system of \'etale covers $\tilde{\Omega}_F^n$ of Drinfeld's rigid-analytic upper half space.  It is now known that in each case, the compactly supported \'etale cohomology $H_c^*(\mathfrak{X})$ realizes the local correspondences on the level of supercuspidal representations of $\GL_n(F)$.  In the vanishing cycle setting this is due to Harris and Taylor~\cite{HarrisTaylor:LLC} in the $p$-adic case and Boyer~\cite{Boyer} in the function field case.  In the rigid setting it is due to Harris~\cite{Harris:Supercuspidal} in the $p$-adic case and Hausberger~\cite{Hausberger} in the function field case.

In each of the above cases the establishment of the correspondences in cohomology begins by embedding the local problem into a global one and appealing to results from the theory of Shimura varieties or Drinfeld modular varieties.   Strauch~\cite{Strauch} proved that the Euler characteristic of $M^n_{\text{LT},F}$ realizes the Jacquet-Langlands correspondence without the use of global moduli spaces.  Also notable is Yoshida's purely local study~\cite{yoshida} of the vanishing cycles of the deformation space of formal $\OO_F$-modules with tame level structure; these are shown to realizes the local Langlands correspondence for supercuspidal representations of depth zero.

Running parallel to these advances in non-abelian Lubin-Tate theory are a great deal of results which give {\em explicit} and {\em purely local} constructions of representations of $\GL_n(F)$ and of $B^\times$, rather than an abstract construction which realizes these representations in cohomology.  Earliest among these is the paper of Howe~\cite{Howe}, which associates a supercuspidal representation of $\GL_n(F)$ to each ``admissible" character of a degree $n$ extension $E/F$.  By elaborating on this construction of supercuspidals for $\GL_2$, Kutzko~\cite{Kutzko:1},~\cite{Kutzko:2} established the local Langlands correspondence for $n=2$.  The fundamental work of Bushnell and Kutzko~\cite{BushnellKutzko} gives an explicit parametrization of admissible representations of $\GL_n(F)$ in terms of their theory of strata.  From here it is natural to attempt to describe the correspondences purely in terms of this parametrization.   This is what is done in the papers of Henniart~\cite{HenniartJLI} and Bushnell-Henniart~\cite{HenniartJLII},~\cite{HenniartJLIII}, where many cases of the Jacquet-Langlands correspondence are established explicitly;  further papers of Bushnell-Henniart~\cite{BushnellHenniart:I},~\cite{BushnellHenniart:II} give an explicit description of the local Langlands correspondence in the ``essentially tame case".

This paper is a modest attempt to draw a connection between the geometry in non-abelian Lubin-Tate theory and the
explicit methods of cuspidal strata and types.  We take $n=2$ and $F$ to have odd residual characteristic $p$.  The aim of this paper is to explicitly construct a variety $\mathfrak{X}$ defined over the algebraic closure of the residue field $k$ of $F$ which plays the role of $M^2_{\text{LT},F}$ from the point of view of non-abelian Lubin-Tate theory.  That is, $\mathfrak{X}$ admits an action of $\GL_2(F)\times B^\times\times W_F$ in such a way that the correspondences are realized in $H^1_{\text{\'et}}(\mathfrak{X},\QQ_\ell)$ for all primes $\ell\neq p$.

The variety $\mathfrak{X}$ is not particularly exotic:  its irreducible components are smooth geometrically connected projective curves over $\overline{k}$, and the only singularities of $\mathfrak{X}$ occur as normal crossings between these components.  The connected components of $\mathfrak{X}$ are in canonical bijection with the connected components of $M^2_{\text{LT},F}$.
For this reason we refer to $\mathfrak{X}$ as the {\em stable Lubin-Tate curve} for $\GL_2(F)$.

A key feature of this construction is that certain subtleties of the local Langlands correspondence now admit a natural explantation in terms of the geometry of $\mathfrak{X}$.  To wit, suppose $E/F$ is a tame quadratic field extension and $\chi$ is a character of $E^\times$, identified with a character of $W_E$, such that $\Ind_{E/F}\chi$ is irreducible.  There is a ``na\"ive" method of constructing a supercuspidal representation $\pi_\chi$ of $\GL_2(F)$, as described in~\cite{Howe}.  It is not the case that $\Ind_{E/F}\chi\mapsto\pi_\chi$ is the Langlands correspondence; {\em e.g.} because the central character of $\pi_\chi$ does not agree with $\det\Ind_{E/F}\chi$ as characters of $F^\times$.  One must modify the na\"ive construction by twisting $\chi$ by a certain tamely ramified character $\Delta_\chi$ of $E^\times$;  then $\Ind_{E/F}\pi_{\chi\Delta_\chi}\mapsto\pi_\chi$ serves as the correct correspondence.  The character $\Delta_\chi$ may be described in an {\em ad hoc} fashion.  In \S\ref{LLCrealization} we show how $\Delta_\chi$ appears in the study of the action of Frobenius on the cohomology of some interesting curves over $k$.

\subsection{The Correspondences}  Let $\Omega$ be an algebraically closed field of characteristic 0.  Let $\mathcal{A}_2(F,\Omega)$ be the set of equivalence classes of smooth irreducible representations of $\GL_2(F)$ with coefficients in $\Omega$, and let $\mathcal{A}_2^d(F,\Omega)$ be the set of discrete series representations.  Likewise, let $\mathcal{A}_2^B(F,\Omega)$ be the set of equivalence classes of smooth irreducible representations of $B^\times$.  We simply write $\mathcal{A}_2^d(F)$ or $\mathcal{A}_2^B(F)$ if $\Omega=\C$.  The {\em Jacquet-Langlands correspondence} is a bijection $$\JL\from \mathcal{A}_2(F)\to\mathcal{A}_2^B(F)$$ satisfying the appropriate trace identity, see~\cite{JacquetLanglands}.  The correspondence $\JL$ is algebraic in the sense that it commutes with field automorphisms of $\C$.  Therefore $\JL$ may be extended canonically to a bijection $\mathcal{A}_2^d(F,\Omega)\to\mathcal{A}_2^B(F,\Omega)$, which we also call $\JL$.

The local Langlands correspondence is {\em not} algebraic;  we therefore work with a slight renormalization.  Let $\ell$ be a prime different from $p$. Let $\mathcal{G}_2(F,\overline{\Q}_\ell)$ be the set of isomorphism classes of Weil-Deligne representations of $W_F$ with coefficients in $\QQ_\ell$.  The {\em $\ell$-adic local Langlands correspondence} is a bijection
$$\mathcal{L}_\ell\from \mathcal{G}_2(F,\overline{\Q}_\ell) \to \mathcal{A}_2(F,\overline{\Q}_\ell)$$ which commutes with automorphisms of the field $\overline{\Q}_\ell$.  It is normalized so that whenever $\iota\from \overline{\Q}_\ell\isom\C$ is a field isomorphism, we have
\begin{align*}
L(\chi\mathcal{L}_\ell(\sigma)^\iota,s)&=L(\chi\sigma^\iota,s-\tfrac{1}{2})\\
\eps(\chi\mathcal{L}_\ell(\sigma)^\iota,s,\psi)&=\eps(\chi\sigma^\iota,s-\tfrac{1}{2},\psi)
\end{align*}
for all representations $\sigma\in\mathcal{G}_2(F,\overline{\Q}_\ell)$, all characters $\chi$ of $F^\times$, and all characters $\psi$ of $F$.  
This is the normalization that appears in the association of Galois representations to Hilbert modular forms.

\subsection{Statement of main theorem}

For our purposes, a {\em stable} curve over $\overline{k}$ is a variety $\mathfrak{X}$ proper and flat over $\Spec \overline{k}$ whose irreducible components are smooth irreducible curves over $k$ such that
\begin{enumerate}
\item The only singularities of $\mathfrak{X}$ are normal crossings between distinct irreducible components, and
\item Each rational component of $\mathfrak{X}$ meets the other components in at least three points.
\end{enumerate}
We do not require that $\mathfrak{X}$ be of finite presentation over $\Spec\overline{k}$.

\begin{defn}
\label{semilineardefn}
Let $G$ be a group admitting a homomorphism $\phi\from G\to \Gal(\overline{k}/k)$.  Let $X$ be a $\overline{k}$-scheme.  An action of $G$ on $X$ is called {\em $\overline{k}$-semilinear} with respect to $\phi$ if for all $t\in G$ the diagram
$$\xymatrix{
X \ar[r]^t \ar[d]  & X \ar[d] \\
\Spec\overline{k} \ar[r]^{\phi(t)} & \Spec\overline{k}
}$$
commutes.
\end{defn}

Our main theorem applies this definition to the triple product group $G=\GL_2(F)\times\times B^\times\times W_F$.  For the homomorphism $\phi\from G\to \Gal(\overline{k}/k)$ we take projection onto $W_F$ followed by the natural map $W_F\to\Gal(\overline{k}/k)$.

\begin{Theorem}  \label{mainthm} Assume that the characteristic of $k$ is $p\neq 2$.  Then there exists a stable curve $\mathfrak{X}$ over $\overline{k}$ admitting a semilinear action of $\GL_2(F)\times B^\times\times W_F$ with the following property.  For every prime $\ell\neq p$ and every supercuspidal representation $\pi$ of $\GL_2(F)$ with coefficients in $\overline{\Q}_\ell$ we have
$$\Hom_{\GL_2(F)}\left(\pi,H^1(\mathfrak{X},\overline{\Q}_\ell)\right)\isom\JL(\pi)\otimes\mathcal{L}(\check{\pi}).$$  On the other hand if $\pi$ is not supercuspidal, then $\Hom_{\GL_2(F)}\left(\pi,H^1(\mathfrak{X},\overline{\Q}_\ell)\right)=0$.
\end{Theorem}

\begin{rmk}  It may be possible that a proof of Thm.~\ref{mainthm} can be given by means of Shimura curves.  Each member of the Lubin-Tate tower $M_{LT,F}^2$ appears as the completion a Shimura curve at a supersingular point.  Over a sufficiently large extension of scalars one can find a model for each Shimura curve which has semi-stable reduction;  by~\cite{Coleman:stable} this can be done in a functorial manner, so that there are maps between the reductions are finite.  In the inverse limit of the reductions, the fiber over a supersingular point ought to have the properties of the stable curve $\mathfrak{X}$ above.  The {\em explicit} determination of the stable reduction of a Shimura curve seems to be quite difficult, however.  This was carried out for the modular curve $X_0(p^2M)$, with $p\geq 5$ and $p\nmid M$ in~\cite{Edix}, and for the modular curve $X_0(p^3M)$ in~\cite{coleman:Xp3}.  Our curve $\mathfrak{X}$ represents our best guess for the structure of the stable reduction of the Shimura ``curve" of infinite $p$-power level.
\end{rmk}

\subsection{Outline of the construction}
\label{outline}
Some preparatory material concerning moduli of deformations of one-dimensional formal groups of height $h$ is given in \S\ref{FormalGroups}.   In \S\ref{CMpoints} we restrict our attention to the case of $h=2$, and discuss deformations with ``CM" (these are Gross' canonical lifts, see~\cite{Gross:canonical}).  For each point $x$ with $CM$ by a tamely ramified quadratic extension $E/F$, we define a decreasing family of subgroups $\mathcal{K}_{x,m}^1\subset \GL_2(F)\times B^\times$, along with certain finite-dimensional representations $\tau$ of $\mathcal{K}_{x,m}^1$.  We prove that the induced representations $\Ind_{\mathcal{K}_{x,m}^1}^{\GL_2(F)\times B^\times}\tau$ realize the Jacquet-Langlands correspondence for the representations of the form $\pi_\chi$, where $\chi$ is an admissible character of $E^\times$ of essential level $m$ (for definitions, see~\ref{admissiblepairs}).  In \S\ref{JLrealization} we define smooth proper curves $\mathfrak{X}_{x,m}$ over $\overline{k}$ admitting an action of $\mathcal{K}_{x,m}^1$ for which the \'etale cohomology $H^1(\mathfrak{X}_{x,m})$ is a direct sum of the representations $\tau$.   In \S\ref{LLCrealization} an action of $W_F$ is introduced so that the fiber product $\mathfrak{X}_{x,m}\times_{\mathcal{K}_{x,m}^1} (GL_2(F)\times B^\times)$ realizes {\em both} correspondences simultaneously in cohomology.  Finally, in \S\ref{construction} the curves $\mathfrak{X}_{x,m}$ are glued together to produce the stable curve $\mathfrak{X}$ in Thm.~\ref{mainthm}.

Suprisingly, there are only two $\overline{k}$-isomorphism classes of higher genus curves which appear among the $\mathfrak{X}_{x,m}$.  One is the Deligne-Lusztig curve for $\SL_2(k)$, with affine equation $XY^q-X^qY=1$.  (It is a special case of~\cite{yoshida} that this curve should appear in the stable reduction of the moduli space of deformations with Drinfeld level-$\varpi_F$ structure.)  The other is the hyperelliptic curve with affine equation $Y^2=X^q-X$, which appears in the stable reduction of the modular curve $X_0(Np^3)$, see~\cite{coleman:Xp3}.  The interplay between the geometry and representation theory of these curves is studied in \S\ref{curves}.



\section{Deformations of one-dimensional formal groups}
\label{FormalGroups}

\subsection{The Moduli Problem.}
Let $F$ be a local nonarchimedean field with uniformizer $\varpi$, maximal ideal $\gp$ and residue field $k$ of cardinality $q$.  Let $n\geq 1$ and let $\Fc_0$ be a one-dimensional formal $\OO_F$-module of height $n$.  Let $\mathcal{C}$ be the category of complete local Noetherian $\hat{\OO}_F^{\text{nr}}$-algebras with residue field $\kk$.   For an integer $m\geq 0$, we consider the moduli problem $\mathcal{M}_m$ which associates to each $R\in\mathcal{C}$ the set of isomorphism classes of {\em deformations} of $\Fc_0$ with $\varpi^m$-level structure.  This is a triple $(\Fc,\iota,\phi)$, where $\F$ is a formal $\OO_F$-module over $R$, $\iota\in\Hom_{\kk}(\Fc_0,\Fc_{\kk})\otimes_{\OO_F} F$ is a quasi-isogeny, and $\phi$ is a {\em Drinfeld level-$m$-structure}, that is, an $\OO_F$-module homomorphism $$\phi\from (\varpi^{-m}\OO_F/\OO_F)^n\to \mathfrak{m}_R,$$ such that the power series $[\varpi]_\Fc(T)\in R\llbracket T \rrbracket$ is divisible by $$\prod_{a\in (\varpi^{-1}\OO_F/\OO_F)^n} (T-\phi(a)).$$  (The maximal ideal $\mathfrak{m}_R$ is to be regarded here as an $\OO_F$-module via $\Fc$.)  An isomorphism between triples $(\Fc,\iota,\phi)$ and $(\Fc',\iota',\phi')$ is an isomorphism of formal $\OO_F$-modules $f\from \Fc\to\Fc'$ which interlaces $\iota$ with $\iota'$ and $\phi$ with $\phi'$.  There are obvious degeneracy maps $\mathcal{M}_{m+1}\to\mathcal{M}_m$.

\subsection{Heights and the division algebra.}
An isogeny $\iota\from\Fc\to\Fc'$ between formal groups has $F$-height $\height_F(\Fc)=h$ if $\ker\iota$ is a group scheme of rank $q^h$ over $\kk$.  If $\iota$ is only a quasi-isogeny, let $r$ be such that $\varpi^r\iota$ is an isogeny and define the $F$-height of $\iota$ as $\height_F(\varpi^r\iota)-rn$.  For a given $h\in \Z$, we may consider the sub-problem $\mathcal{M}_m^{(h)}\subset\mathcal{M}_m$ of deformations of $\Fc_0$ for which the quasi-isogeny has $F$-height $h$.  Then $$\mathcal{M}_m=\coprod_{h\in\Z}\mathcal{M}_m^{(h)}.$$

Let $\OO_B=\End_{\kk}\Fc_0$.  Then $\OO_B$ is the unique maximal compact subring of $B=\OO_B\otimes_{\OO_F} F$, which is in turn a division algebra over $F$ with invariant $1/n$.  There is a right action of $B^\times$ on $\mathcal{M}_m$ given by $(\Fc,\iota,\phi)^b= (\Fc,\iota\circ b,\phi)$ for $b\in B^\times$.  Let $N\from B\to F$ be the reduced norm.  Since $b\from\Fc_0\to\Fc_0$ has $F$-height $v(N(b))$, we see that the action of $b$ maps $\mathcal{M}_m^{(h)}$ isomorphically onto $\mathcal{M}_m^{(h+v(N(b)))}$.  (Here $v$ is the valuation on $F^\times$ with $v(\varpi)=1$.)

\subsection{The associated rigid spaces $M_m$.}  It is a result of Drinfeld that $\mathcal{M}_m^{(0)}$ is representable by a regular local ring $R_m=R_m^{(0)}$ of dimension $n$, that each map $R_m\to R_{m+1}$ is finite and flat and \'etale over the generic fiber, and that $R_0$ is isomorphic to $\OO_F^{\text{nr}}\llbracket X_1,\dots,X_{n-1} \rrbracket$.  Similarly, each $\mathcal{M}_m^{(j)}$ is representable by a regular local ring $R_m^{(h)}$.  Therefore $\mathcal{M}_m$ has the structure of a formal scheme which is locally formally of finite type over $\hat{\OO}_F^{\text{nr}}$.
Let $M_m$ be the rigid space attached to the generic fiber of $\mathcal{M}_m$.  The morphisms $M_{m+1}\to M_m$ are \'etale and the space $M_0$ is the union of spaces $M_0^{(j)}$ for $j\in \Z$, each of which is the rigid open unit polydisk of dimension $n-1$.  The spaces $M_m$ inherit an action of $B^\times$.

\subsection{The limit problem, action of $\GL_n(F)$.}  The moduli problem $\mathcal{M}_m$ has a right action of $\GL_n(\OO_F/\gp^m)$ given by $(\Fc,\iota,\phi)^g=(\Fc,\iota,\phi\circ g)$.  Therefore the rigid analytic space $M_m$ admits an action of $\GL_n(\OO_F/\gp^m)$.  These actions coalesce into an action of $G=\GL_n(F)$ on the projective system $$M=\lim_{\infty\leftarrow m} M_m.$$  To describe this action, we give an alternate description of $M$.  For the time being, let $M'$ be the functor which assigns to each complete subfield $K\subset \C_F$ containing $\hat{F}^{\text{nr}}$ the set of isomorphism classes of triples $(\Fc,\iota,\alpha)$, where $\Fc/\OO_K$ is a formal $\OO_F$-module ``up to isogeny", $\iota\in \Hom_{\kk}(\Fc_{\kk},\Fc_0)\otimes_{\OO_F}F$ is a quasi-isogeny, and $\alpha\from F^n\to V(\Fc)$ is an isomorphism of $F$-vector spaces.  Here $V(\Fc)=T(\Fc)\otimes_{\OO_F} F$, and $$T(\Fc)=\lim_{\leftarrow} \Fc[\varpi^m](K)$$ is the Tate module.  An isomorphism between two triples $(\Fc,\iota,\alpha)$ and $(\Fc',\iota',\alpha')$ is a {\em quasi-isogeny} $f\in \Hom_{\OO_K}(\Fc,\Fc')\otimes_{\OO_F} F$ carrying $\iota$ to $\iota'$ and $\phi$ to $\phi'$.  We shall call $M'$ the functor of deformations {\em up to isogeny}.

We claim that $M'$ agrees with the functor of points on $M$.  Indeed, suppose a point of $M(K)$ is represented by an inductive system of points $(\Fc,\iota,\phi_m)_{m\geq 0}$ of $M_m(K)$.  Then the Drinfeld level-$m$-structures $\phi_m$ give rise to an isomorphism of $\OO_F$-modules $\alpha_0\from \OO_F^n\to T(\Fc)$ in an evident way.  Let $\alpha=\alpha\otimes 1$ be the extension of this map to an isomorphism $F^n\to V(X)$.  Then $(\Fc,\iota,\alpha)\in M'(K)$.

Going the other way, suppose $(\Fc,\iota,\alpha)$ represents a point of $M'(K)$.  Let $\mathcal{L}=\phi(\OO_F^n)$;  this is a lattice in $V(\Fc)$.   Let $r\in \Z$ be such that $T(\Fc)\subset \varpi^r\mathcal{L}$.  We have an exact sequence $$0\to T(\Fc)\to V(\Fc)\to \Fc[\varpi^\infty]\to 0;$$ let $C$ be the image of $\varpi^r\mathcal{L}$ in $\Fc[\varpi^\infty]$, so that $C\isom\varpi^r\mathcal{L}/T(\Fc)$.  Write $\Fc'$ for the quotient $\Fc/C$, and write $f\in \Hom(\Fc,\Fc')\otimes_{\OO_F} F$ for the quasi-isogeny
$$\xymatrix{
\Fc\ar^{\varpi^r}[r]& \Fc \ar[r]& \Fc/C=\Fc'}.$$  Then we have the commutative square
$$\xymatrix{
V(\Fc) \ar^{V(f)}[r] & V(\Fc') \\
\mathcal{L} \ar@{^{(}->}[u] \ar^{\sim}[r] & T(\Fc') \ar@{^{(}->}[u]
}$$
The isomorphism $\alpha\from \OO_F^n\to \mathcal{L}$ induces isomorphisms $\phi_m\from (\varpi^{-m}\OO_F/\OO_F)^n\to \Fc[\varpi^m](K)$ for each $m\geq 0$.  Let $\iota'=\iota\circ f_{\kk}^{-1}\in\Hom_{\kk}(\Fc_{\kk}',\Fc_0)\otimes_{\OO_F} F$.  Then $f$ gives an isomorphism between the triples $(\Fc,\iota,\alpha)$ and $(\Fc',\iota',V(f)\circ\alpha)$.  The latter triple corresponds in turn to the inductive system $(\Fc',\iota',\phi_m)_{m\geq 0}$.   We conclude that $M=M'$.

It is now evident how to define a right action of $G=\GL_2(F)$ on the rigid space $M$:  Given a triple $(\Fc,\iota,\alpha)$ as in the definition of $M'$ and a matrix $g\in G$, define $(\Fc,\iota,\alpha)\cdot g=(\Fc,\iota,\alpha\circ g)$.  Note that $g$ maps $M^{(h)}$ isomorphically onto $M^{(h-v(\det g))}$.

\subsection{Action of $W_F$} \label{actionofwf} Let $W_F$ be the Weil group of $F$.  Recall that the rigid space $M$ is defined over $\hat{F}^{\text{nr}}$.  There is an action of $W_F$ on $M$ lying over the action of $W_F$ on $\hat{F}^{\text{nr}}$.  We describe this action on the level of $\C_F$-points.  Let $\sigma\in W_F$ and let $x=(\Fc,\iota,\alpha)\in M(\C_F)$.  It is clear how to define the conjugate formal group $\Fc^\sigma$ and level structure $\alpha^\sigma$.  We must now construct a quasi-isogeny $\Fc^\sigma_{\kk}\to\Fc_0$.  Let $W_F\to \Gal(\overline{k}/k)$ be the obvious map, and assume the image of $\sigma$ equals $\Fr^n$ for some $n\geq 0$, where $\Fr$ is the $q$th power map on $\overline{k}$.  Start with the quasi-isogeny $\iota^\sigma\from \Fc^\sigma_{\overline{k}}\to\Fc_0^{\Fr^n}.$  To get a quasi-isogeny with values in $\Fc_0$, we compose this map with the inverse of the natural isogeny $\Fc_0\to\Fc_0^{\Fr^n}$ given by $X\mapsto X^{q^n}$.  This action of $W_F$ commutes with the actions of $\GL_2(F)$ and $B^\times.$

We therefore have an action of $G\times B^\times \times W_F$ on $M$.  Define a homomorphism
\begin{eqnarray*}
\delta\from G\times B^\times \times W_F &\to& F^\times  \\
(g,b,w)&\mapsto& (\det g)^{-1}\times N(b)\times \Art_F^{-1}w,
\end{eqnarray*}
where $\Art_F\from F^\times \to W_F^{\text{ab}}$ is the reciprocity map from local class field theory, normalized so that $\Art_F(\varpi)$ is a geometric Frobenius element.  Then a triple $(g,b,w)\in G\times B^\times \times W_F$ maps $M^{(h)}$ onto $M^{(h+d)}$, where $d=v_F(\delta(g,b,w))$.

\subsection{Connected Components.}  Let $\mathcal{Z}$ be a one-dimensional $\OO_F$-module of height 1 over $\hat{\OO}_{F^{\text{nr}}}$.  By classical Lubin-Tate theory, $\mathcal{Z}$ is unique up to isomorphism.  We summarize here the results of Strauch~\cite{Strauch:ConnComp} concerning the geometrically connected components of the rigid spaces $M_m\otimes\C_F$.

\begin{Theorem} There exists a bijection from $\pi_0(M_m\otimes\C_F)$ onto the set of bases for the rank 1 $(\OO_F/\gp_F^m)$-module $\mathcal{Z}[\varpi^m]$.  This bijection is equivariant for the action of $\GL_n(\OO_F/\gp_F^m)\times \OO_B^\times\times I_F$ if we let an element $(g,b,w)$ act on $\mathcal{Z}[\varpi^m]$ through the homomorphism $(g,b,w)\mapsto\delta(g,b,w)\pmod{\gp^m}$.
\end{Theorem}

Combining this theorem for all $m\geq 1$ gives a description of the set of geometrically connected components of the tower $M$ of rigid spaces:

\begin{Theorem} There exists a bijection $\pi_0(M\otimes \C_F)\isom V(\mathcal{Z})\backslash\set{0}$.  This bijection is equivariant for the action of $\GL_2(F)\times B^\times \times W_F$ if we let an element $(g,b,w)$ act on $V(\mathcal{Z})$ through the homomorphism $\delta\from GL_2(F)\times B^\times\times W_F\to F^\times$.
\end{Theorem}

For a nonzero $\zeta\in V(\mathcal{Z})$, let $M^\zeta\subset M\otimes \C_F$ be the connected component corresponding to $\zeta$.  There is a natural ``valuation" $h\from V(\mathcal{Z})\backslash\set{0}\to \Z$ defined as the least $h\in\Z$ for which $\zeta\in \varpi^hT(\mathcal{Z})$.  Then $M^\zeta\subset M^{h(\zeta)}\otimes\C_F$.

\section{CM points and the Jacquet-Langlands correspondence}
\label{CMpoints}

\subsection{CM Points: Basic Observations.}
It is at this point that we restrict our attention to the case $n=2$.  We abbreviate $A=M_2(F)$;  $B$ is the nonsplit quaternion algebra over $F$.

A deformation $\Fc$ of $\Fc_0$ to $\C_F$ has {\em CM} if $E=\End \Fc\otimes_{\OO_F} F$ is a quadratic field extension of $F$.  Suppose a point $x\in M$ is represented by a deformation $(\Fc,\iota,\alpha)$ up to isogeny.  We say that $x$ has CM by $E$ if $\End \Fc\otimes_{\OO_F} F=E$.  Note that by replacing $\Fc$ by an isogenous formal $\OO_F$-module we may assume that $\End \Fc=\OO_E$.  In that case, $\Fc$ becomes a formal $\OO_E$-module of height 1.  By classical Lubin-Tate theory, there is only one such $\Fc$ up to isomorphism: let this be called $\Fc_E$.   Note that $\Fc_E$ is defined over $\hat{E}^{\text{nr}}$.  Let $M^E\subset M(\overline{F})$ be the set of all points of $M$ with CM by $E$.

If $x=(\Fc_E,\iota,\alpha)\in M^E$, then we naturally have at our disposal embeddings of $E$ into both $A$ and $B$. Indeed, since $V(\Fc)$ is naturally an $E$-vector space of dimension one, we may identify $E$ with a subfield of $\End V(\Fc)$.  On the other hand, $\alpha$ identifies $\End V(\Fc)$ with $A$, so that there is a unique embedding $j_{x,A}\from E\injects A$ for which the identity $\alpha\circ j_{x,A}(e)=V(e)$ holds in $\End V(\Fc)$ for all $e\in E$.  Similarly, $\iota$ gives a unique embedding $j_{x,B}\from E\to B$ for which the appropriate diagram commutes.  Let $j_{x}\from E\to A\times B$ be the diagonal embedding $e\mapsto (j_{x,A}(e),j_{x,B}(e))$.

The following theorem can be deduced from~\cite{Gross:canonical}:
\begin{Theorem}
\label{CMpointbasic}
\mbox{}
\begin{itemize}
\item[(i)] The group $\GL_2(F)\times B^\times$ acts transitively on $M^E$.  The stabilizer in $\GL_2(F)\times B^\times$ of the point $x$ is $j_x(E^\times)$.
\item[(ii)] For all $t\in\GL_2(F)\times B^\times$ we have $j_{x^t}(\beta)=t^{-1}j_x(\beta)t$, all $\beta\in E$.
\end{itemize}
\end{Theorem}

\subsection{Action of $W_F$ on CM Points.} \label{relativeweil}
Recall that the relative Weil group $W_{E/F}$ is the quotient of $W_F$ by the closure of the commutator subgroup of $W_E$.  Thus $W_{E/F}$ is the preimage of $\Z$ in the surjection $\Gal(E^{\text{ab}}/F)\to\Gal(\overline{k}/k)\isom\hat{\Z}$. There is a short exact sequence
\begin{equation}
1 \to E^\times \to W_{E/F} \to \Gal(E/F) \to 1
\end{equation}
representing the fundamental class in $H^2(\Gal(E/F),E^\times)$.  In the interpretation of this $H^2$ as the relative Brauer group of the extension $E/F$, the fundamental class corresponds to the class of $B$ as a central simple algebra over $F$ which is split by $E$.

By classical Lubin-tate theory, adjoining the torsion points of the CM formal group $\Fc_E$ to $E^{\text{nr}}$ yields the maximal abelian extension $E^{\text{ab}}$ of $E$.   Thus the action of $W_F$ on $M^E$ factors through $W_{E/F}$.  We make this action explicit.  Let $x\in M^E$, so that we have embeddings $j_{x,A}$ and $j_{x,B}$ of $E$ into $A$ and $B$, respectively.  Let $\mathcal{N}_{x,A}$ (resp., $\mathcal{N}_{x,B}$) be the normalizer of $j_{x,A}(E^\times)$ in $\GL_2(F)$ (resp., the normalizer of $j_{x,B}(E^\times)$ in $B^\times$).  Then $\mathcal{N}_{x,A}$ and $\mathcal{N}_{x,B}$ are both extensions of $\Gal(E/F)$ by $E^\times$, representing the trivial and nontrivial classes in $H^2(\Gal(E/F), E^\times)$, respectively.  Let $\mathcal{N}_x\subset \GL_2(F)\times B^\times$ be the pullback group in the diagram
$$
\xymatrix{\mathcal{N}_x \ar[r] \ar[d] & \mathcal{N}_{x,A}  \ar[d] \\
\mathcal{N}_{x,B} \ar[r] & \Gal(E/F).
}
$$
Observe that $j_x(E^\times)\subset\mathcal{N}_x$ embeds as a normal subgroup.  The homomorphism $\delta\from \mathcal{N}_x\to F^\times$ factors through $j_x(E^\times)\backslash \mathcal{N}_x$.  Finally note that $j_x(E^\times) \backslash\mathcal{N}_x$, like $W_{E/F}$, is a nonsplit extension of $\Gal(E/F)$ by $E^\times$.  The following proposition, which is a straightforward application of classical Lubin-Tate theory, gives an natural isomorphism between these two groups.

\begin{prop}  For each $x\in M^E$, there is an isomorphism $j_{x,W}\from W_{E/F} \to j_x(E^\times)\backslash\mathcal{N}_x$ with the following properties:
\begin{itemize}
\item[(i)] For all $w\in W_{E/F}$, $x^w = x^{j_{x,W}(w)}$.
\item[(ii)] For an element $\beta\in E=W_E^{\text{ab}}$ we have $$j_{x,W}(w)=j_x(E^\times)(j_{x,A}(\beta),1)=j_x(E^\times)(1,j_{x,B}(\beta)^{-1}).$$
\item[(iii)] For $w\in W_{E/F}$, we have $\delta(j_{x,W}(w))=\Art_F^{-1} w$.
\item[(iv)] For $t\in \GL_2(F)\times B^\times$ and all $w\in W_{E/F}$ we have $j_{x^t,W}(w)=tj_{x,W}(w)t^{-1}$.
\end{itemize}
\end{prop}

\subsection{Lattice Chains and Chain Orders.} \label{latticechains}  Given $x\in M^E$ as above, identify $E$ with a subfield of $A$ via $j_{x,A}$. We consider the sequence $(\Lambda_i)_{i\in \Z}$ of lattices in $F^2$ defined by
 $$\Lambda_i=\alpha^{-1}\left(\gp_E^iT(\Fc_E)\right).$$   Then $(\Lambda_i)$ is an {\em $\OO_F$-lattice chain} in the sense that this collection of lattices is linearly ordered and stable under multiplication by $F^\times$.   Let $\mathfrak{A}_x\subset M_n(F)$ be the subalgebra of matrices which stabilize {\em each} $\Lambda_i$.  We drop the subscript $x$ from the notation when $x$ is fixed in the discussion.   Then $\mathfrak{A}$ is normalized by $E^\times$. Let $\gP\subset \mathfrak{A}$ be the double-sided ideal of matrices mapping $\Lambda_i$ into $\Lambda_{i+1}$ for all $i$.  We remark
 $\mathfrak{A}$ is conjugate to $M_2(\OO_F)$ or $\tbt{\OO_F}{\OO_F}{\gp_F}{\OO_F}$, as $E/F$ is unramified or ramified, respectively.

\subsection{Admissible Pairs}
\label{admissiblepairs}
Assume $E/F$ is a tamely ramified quadratic extension, and that $\chi$ is a character of $E^\times$.
The pair $(E/F,\chi)$ is called {\em admissible} if (1) $\chi$ does not
factor through the norm map $N_{E/F}$, and if (2) $\chi\vert_{U_E^1}$ does not factor through the norm map if $E/F$ is ramified.  A character $\chi$ has {\em level $m$} if it vanishes on $1+\gp_E^{m+1}$ but not on $1+\gp_E^m$;  we say it has {\em essential level $m$} if the minimum level of the characters of the form $\chi\times(\omega\circ N_{E/F})$ is $m$.

We will have use for a simple parametrization of admissible pairs.  First, if $E/F$ is unramified and $\theta$ is a character of $k_E^\times$ for which $\theta\neq\theta^q$, then let $X_\theta$ denote the set of characters $\chi'$ of $E^\times$ of the form $\chi'=\chi\times (\omega\circ N_{E/F})$, where $\chi$ is a character of $E^\times$ whose restriction to $\OO_E^\times$ is inflated from $\theta$, and where $\omega$ is a character of $F^\times$.  Then characters of $X_\theta$ have essential level zero, and every character of essential level zero belongs to some $X_\theta$.  We have that $X_\theta$ and $X_{\theta'}$ are disjoint unless $\theta'\theta^{-1}$ factors through $N_{k_E/k}$, in which case $X_\theta=X_{\theta'}$.

Characters of higher level may be parametrized by certain elements of $E^\times$:

\begin{defn} Let $m\geq 1$.
An element $\alpha\in E^\times$ of valuation $v_E(\alpha)=-m$ is {\em minimal} when it satisfies the criteria:
\begin{itemize}
\item[(i)] If $E/F$ is ramified, then $m$ must be odd.
\item[(i)] If $E/F$ is unramified, then the minimal polynomial of $\varpi_F^m\alpha$ over $F$ is irreducible modulo $\gp_F$.
\end{itemize}
\end{defn}

Fix a character $\psi$ of $F$ of level zero.  Let $\psi_E$ be the character $x\mapsto\psi(\tr_{E/F}x)$;  since $E/F$ is tamely ramified, $\psi_E$ is a character of $E$ of level zero.

Let $m\geq 1$.  If $\alpha\in E^\times$ is a minimal element of valuation $-m$, let $X_\alpha$ be the set of characters $\chi$ of $E^\times$ of the form $\chi'=\chi\times(\omega\circ N_{E/F})$, where $\chi'$ is a character of level $m$ satisfying $\chi(1+x)=\psi_E(\alpha x)$, all $x\in\gp_E^{\floor{m/2}+1}$.  Then each character $\chi\in X_\alpha$ is admissible of essential level $m$.  Every admissible character of essential level $m$ belongs to some $X_\alpha$.


If $(E/F,\chi)$ is an admissible pair, then $\Ind_{E/F}\chi\in\mathcal{G}_2(F)$ is irreducible.
There is a straightforward way of attaching a supercuspidal representation $\pi_\chi$ of $\mathcal{A}_2(F)$ to each admissible pair $(E/F,\chi)$, which already appears in the general case of $\GL_n(F)$ in~\cite{Howe}.

We sketch this construction.  Assume there is no character $\omega$ of $F^\times$ for which $\chi\times\omega\circ N_{E/F}$ has smaller level than $\chi$.  Choose an embedding $E^\times\injects \GL_2(F)$.  There are three cases to consider:

\begin{itemize}
\item $\chi$ has level 0.  Suppose $\chi\in X_\theta$ for a character $\theta$ of $k_E^\times$.  Let $\lambda_\theta$ be the associated cuspidal representation of $\GL_2(\theta)$, inflated to $\GL_2(\OO_F)$.  Let $J_A=E^\times\GL_2(\OO_F)$ and let $\Lambda_\chi$ be the representation of $J$ which extends $\lambda_\theta$ and which agrees with $\chi$ on $F^\times$.
\\

\item $\chi$ has level $m=2r-1>0$.  Suppose $\chi\in X_\alpha$.    Let $J_A=E^\times U_{\mathfrak{A}}^r$.  Define a character $\Lambda_\chi$ of $J$ which agrees with $\chi$ on $E^\times$ and which satisfies $\Lambda(1+x)=\psi_A(\alpha x)$ for $x\in \gP_{\mathfrak{A}}^r$.
\\

\item $\chi$ has level $m=2r$.  Then $E/F$ must be unramified.  Suppose $\chi\in X_\alpha$.  Define a character $\theta$ of  $H=U_E^1U_{\mathfrak{A}}^r$ which agrees with $\chi$ on $U_E^1$ and which satisfies $\theta(1+x)=\psi_A(\alpha x)$ for $x\in \gP_{\mathfrak{A}}^{r+1}$.  Let $J^1=U_E^1U_\mathfrak{A}^{r}$;  then there exists a unique irreducible representation $\eta_\theta$ of $J^1$ of dimension $q$ lying over $\theta$.  Finally let $J_A=E^\times U_\mathfrak{A}^r$;  there is a unique extension $\Lambda_\chi$ of $\eta_\theta$ to $J_A$ which lies over $\eta_\theta$ and which satisfies $\tr\Lambda_\chi(\zeta)=-\chi(\zeta)$ for all roots of unity $\zeta\in\mu_E\backslash\mu_F$.

\end{itemize}

In all cases let $\pi_\chi=\Ind_{J_A}^{\GL_2(F)}\Lambda$.  Then $\pi_\chi$ is a supercuspidal representation of $\GL_2(F)$, and the correspondence $(E/F,\chi)\mapsto\pi_\chi$ is a bijection from the set of equivalence classes of admissible pairs onto the set of tame supercuspidal representations of $\GL_2(F)$.

There is a similar procedure to construct a representation $\pi_\chi'$ of $B^\times$.  Choose an embedding $E\injects B$.  Identify $\OO_B/\gP_B$ with $k_E$ in a way which is compatible with this embedding. Once again, there are three cases to consider:

\begin{itemize}

\item $\chi$ has level 0.  Suppose $\chi\in X_\theta$.  Let $J_B=E^\times \OO_B^\times$ and let $\Lambda_\chi'$ be the character of $J_B$ which agrees with $\chi$ on $E^\times$ and which is the inflation of $\theta$ on $\OO_B^\times$.
\\
\item $\chi$ has level $m=2r-1$ and $E/F$ is ramified, or else $m=2r$ and $E/F$ is unramified.  Suppose $\chi\in X_\alpha$.   Let $f$ be the residual degree of $E/F$.  Let $J_B=E^\times U_B^{rf}$.  Define a character $\Lambda'_\chi$ of $J$ which agrees with $x\mapsto (-1)^{fv_E(x)}\theta(x)$ on $E^\times$ and which satisfies $\Lambda'_\chi(1+x)=\psi_A(\alpha x)$ for $x\in \gP_{B}^{rf}$.
\\
\item $\chi$ has odd level $m$ and $E/F$ is unramified.  Suppose $\chi\in X_\alpha$.  Define a character $\theta$ of $H=U_E^1U_{B}^{m+1}$ which agrees with $\chi$ on $U_E^1$ and which satisfies $\theta(1+x)=\psi_B(\alpha x)$ for $x\in\gP_{B}^{m+1}$.  Let $J^1=U_E^1U_B^{m}$;  then there exists a unique irreducible representation $\eta_\theta$ of $J^1$ of dimension $q$ lying over $\theta$.  Finally let $J_B=E^\times U_B^{m}$;  there is a unique extension $\Lambda_\chi'$ of $\eta_\theta$ to $J_B$ which lies over $\eta_\theta$ and which satisfies $\tr\Lambda_\chi(\zeta)=-\Lambda_\chi(\zeta)$ for all roots of unity $\zeta\in\mu_E\backslash\mu_F$.

\end{itemize}

In all cases let $\pi_\chi'=\Ind_{J_B}^{B^\times}\Lambda'_\chi$.  The following is from~\cite{Henniart:Bushnell}, \S 56:

\begin{Theorem} For all admissible pairs $(E/F,\chi)$, we have $\JL(\pi_\chi)=\pi_{\chi}'$.
\end{Theorem}


\subsection{Filtrations of $\mathfrak{A}$ and $\OO_B$ by $\OO_E$-submodules}
Once again, $A=M_2(F)$ and $E/F$ is a quadratic extension field.  We fix an embedding $E\injects A$.  

A {\em tame corestriction} on $A$ relative to $E/F$ is an $(E,E)$-bimodule homomorphism $s_A\from A\to E$ such that $s(\mathfrak{A})=\OO_E$ for any hereditary $\OO_F$-order in $A$ which is normalized by $E^\times$.  A tame corestriction exists and is unique up to multiplication by $\OO_E^\times$, see~\cite{BushnellKutzko}.  It further satisfies $s_A(\gP^r)=\gp_E^r$ for $r\geq 1$.

In the event that $E/F$ is a tame extension, there is a simple description of $s$.  Let $C$ be the complement of $E$ under the trace pairing $A\times A\to F$, so that $A=E\oplus C$.  Then $s_A$ is the projection of $A$ onto $E$ with respect to this decomposition.

For an integer $m\geq 0$, define $V^m_A\subset\mathfrak{A}$ to be the $\OO_E$-submodule
$$V^{m}_A=s_A^{-1}(\gp_E^m)\cap \gP^{\floor{(m+1)/2}}.$$  We remark that $V^0_A=\mathfrak{A}$.  The module $V^m_A$ is closed under multiplication, and in fact no smaller choice of exponent $\floor{(m+1)/2}$ allows for this property.  Consequently the set $1+V^m_A$ is a
{\em subgroup} of $\mathcal{A}^\times$.

Carrying this idea further, define an $\OO_E$-submodule $W_A^m\subset V_A^m$ by $$W_A^m=s_A^{-1}(\gp_E^m)\cap\gP^{\floor{m/2}+1}.$$  Then $W_A^m$ is closed under multiplication and also enjoys the property:
\begin{equation}\label{goodenough}
s(W_A^mW_A^m)\subset\gp_E^{m+1}.
\end{equation}

Similar constructions can be carried out on the division algebra $B$.  Given $x\in M^E$, identify $E$ with a subfield of $B$ via $j_{x,B}$.  Let $\OO_B$ be the ring of integers in $B$ and let $\gP_B$ be its maximal ideal.  A tame corestriction $s_B\from B\to E$ is an $(E,E)$-bimodule homomorphism such that $s_B(\OO_B)=\OO_E$ and $s(\gP_B^r)=\gp_E^{\ceil{r/f}}$ for all $r\geq 1$, where $f$ is the degree of the extension of residue fields $k_E/k$.

Once again, if $E/F$ is tamely ramified, then $s_B$ is the projection of $B$ onto $E$ relative to the reduced trace $\tr_{B/F}$.

For $m\geq 0$, we define
\begin{eqnarray*}
V_B^{m}&=&s^{-1}(\gp_E^m)\cap \gP_B^{r}\\
W_B^m&=&s^{-1}(\gp_E^m)\cap \gP_B^{r'}
\end{eqnarray*}
where $r$ is the least nonnegative integer for which the module $V_B^m$ so defined is closed under multiplication, and $r'$ is the least nonnegative integer for which the module $W_B^m$ so defined satisfies $s_B(W_B^mW_B^m)\subset \gp_E^{m+1}$.
Explicitly:  The value of $r$ is $2\floor{m/2}$ if $E/F$ is unramified and $\floor{(m+1)/2}$ otherwise.   The value of $r'$ is $m+1$ if $E/F$ is unramified and $\floor{m/2}+1$ if $E/F$ is ramified.

We collect some trivial bookkeeping results:
\begin{prop}  \label{bookkeeping} The dimension of the $k_E$-vector spaces $V_A^m/W_A^m$ and $V_B^m/W_B^m$ are given by the following table:
\begin{center}
\begin{tabular}{ccccc}
 & \multicolumn{2}{c}{$E/F$ ramified} &  \multicolumn{2}{c}{$E/F$ unramified} \\
 & $m$ odd & $m$ even & $m$ odd & $m$ even \\
$\dim V_A^m/W_A^m$ & 0 & 1 & 0 & 1 \\
$\dim V_B^m/W_B^m$ & 0 & 1 & 1 & 0
\end{tabular}
\end{center}
On the other hand, the $k_E$-vector spaces $W_A^m/V_A^{m+1}$ and $W_B^m/V_B^{m+1}$ are always 1-dimensional.
\end{prop}






\subsection{Certain subgroups of $\GL_2(F)\times B^\times$}  In this section we build the subgroups $\mathcal{K}_{x,m}^1$ mentioned in \S\ref{outline}.  Assume $E/F$ is a tame quadratic extension with ramification degree $e$. Let $x\in M^E$ be given.  Identify $E$ with a subfield of $M_2(F)\times B$ via the embedding $j_x$.

For $m\geq 0$, define the product vector spaces $V^m=V_A^m\times V_B^m$, $W^m=W_A^m\times W_B^m$.  Since $V^m$ is closed under multiplication in $A\times B$, $1+V^m$ is a subgroup of $\GL_2(F)\times B^\times$.  The subgroup $1+V^{m+1}\subset 1+V^m$ is normal; we let $\Gr_m$ be the quotient.  

By Prop.~\ref{bookkeeping}, we have $\dim_{k_E}W^m/V^{m+1}=2$, whereas $$\dim_{k_E} V^m/W^m=\begin{cases} 0,& \text{$E/F$ ramified and $m$ odd} \\ 2,&\text{$E/F$ ramified and $m$ even}\\1,& \text{$E/F$ unramified.}\end{cases}$$ When $m\geq 1$, $\Gr_m$ is a two-step nilpotent group, for there is an exact sequence 
\begin{equation}\label{VWsequence}
0\to \frac{W^m}{V^{m+1}}\to \Gr_m \to\frac{V^m}{W^m}\to 0.
\end{equation}

In \S\ref{actionofwf} we defined a homomorphism $\delta\from \GL_2(F)\times B^\times\to W_F\to F^\times$;  in this section we will use the same letter $\delta$ to denote the restriction of this homomorphism to $\GL_2(F)\times B^\times$.  Let $\Gr_m^1$ denote the kernel of the homomorphism $\Gr_m\to\delta(1+V^m)/\delta(1+V^{m+1})$ induced by   $\delta$.  By analyzing the effect of $\delta$ on the vector spaces on either side of $\Gr_m$ in~\eqref{VWsequence} we arrive at:
\begin{lemma} \label{GR1} Assume $m\geq 1$.
\begin{itemize}
\item[(i)]  If $E/F$ is ramified and $m$ is odd, then $\Gr_m^1\isom W^m/V^{m+1}$ is a $k$-vector space of dimension 2.
\item[(ii)] If $E/F$ is ramified and $m$ is even, then $\Gr_m^1\isom V^m/V^{m+1}$ is a $k$-vector space of dimension 2.
\item[(iii)] If $E/F$ is unramified, then $\Gr_m^1$ is a nontrivial extension of $V^m/W^m$ by a $k$-vector subspace of $W^m/V^{m+1}$ of dimension 1.
\end{itemize}
\end{lemma}

We now define subgroups $\mathcal{K}_{x,m}$, $\mathcal{P}_{x,m}$, $\mathcal{L}_{x,m}$ of $\GL_2(F)\times B^\times$ by
\begin{eqnarray*}
\mathcal{K}_{x,m}&=&E^\times (1+V^m) \\
\mathcal{P}_{x,m}&=&F^\times U_E^1 (1+W^m) \\
\mathcal{L}_{x,m}&=&F^\times U_E^1 (1+V^{m+1})
\end{eqnarray*}
Then $\mathcal{K}_{x,m}\supset\mathcal{P}_{x,m}\supset\mathcal{L}_{x,m}$.  
Let $\mathcal{K}_{x,m}^1$ be the kernel of the homomorphism $\delta$ restricted to $\mathcal{K}_{x,m}$, and similarly for $\mathcal{L}_{x,m}^1$.  

We observe the following facts concerning $\mathcal{K}_{x,m}$ and $\mathcal{L}_{x,m}$.  
\begin{prop}\label{Kxm} 
\mbox{}
\begin{itemize}
\item[(i)] For $t\in\GL_2(F)\times B^\times$ we have $\mathcal{K}_{x^{t},m}=t^{-1}\mathcal{K}_{x,m}t$, and similarly for $\mathcal{K}_{x,m}^1$.
\item[(ii)] $\mathcal{L}_{x,m}$ is normal in $\mathcal{K}_{x,m}$.
\item[(iii)] There is a split exact sequence $1\to \Gr_m^1\to \mathcal{K}_{x,m}^1/\mathcal{L}_{x,m}^1\to E^\times/F^\times U_E^1\to 1$.
\item[(iv)] $\mathcal{N}_x$ normalizes the groups $\mathcal{K}_{x,m}$ and $\mathcal{L}_{x,m}$.
\end{itemize}
\end{prop}
We will gather more information on the groups $\mathcal{K}_{x,m}^1$ and $\mathcal{K}_{x,m}^1/\mathcal{L}_{x,m}^1$ in \S\ref{JLrealization}.

\subsection{Characters of $\mathcal{K}$}
\label{charactersofK}

In this paragraph we will see why the groups $\mathcal{K}_{x,m}^1$ are important:  The quotients $\mathcal{K}_{x,m}^1/\mathcal{L}_{x,m}^1$ admit special irreducible characters whose induction to $\GL_2(F)\times B^\times$ realizes the Jacquet-Langlands correspondence for exactly those supercuspidal representations of $\GL_2(F)$ arising from admissible pairs $(E/F,\chi)$ for which $\chi$ has essential level $m$.  

Assume $m\geq 1$, let $E/F$ be a tame quadratic extension field, let $x\in M^E$, and identify $E$ with a subfield of $A\times B$ via $j_x$.  Let $\alpha\in E^\times$ be a minimal element of valuation $-m$.  This forces $m$ to be odd if $E/F$ is ramified.  Define an $(E,E)$-linear map $s\from M_2(F)\times B\to E$ by
$$s(a,b)=s_A(a)-s_B(b),$$ so that $s$ vanishes on $E$.  Note that $s(W^mW^m)\subset\gp_E^{m+1}$.

Write $\mathcal{K}$ for $\mathcal{K}_{x,m}$, and similarly for $\mathcal{L}$ and $\mathcal{P}$.

Let $\psi_\alpha$ be the character of $\mathcal{P}$ defined by the rules\begin{eqnarray*}
\psi_\alpha(F^\times U_E^1)&=&1\\
\psi_\alpha(1+w)&=&\psi_E(\alpha s(w)),\;w\in W^m
\end{eqnarray*}
This is well-defined because $s(W^m W^m)\subset\gp_E^{m+1}$.  Note that $\psi_\alpha$ vanishes on $\mathcal{L}$.

We know define a certain irreducible representation $\tau_\alpha$ of $\mathcal{K}$ which lies over $\psi_\alpha$.  If $E/F$ is ramified, then $\mathcal{K}=E^\times\mathcal{P}$.  In this case, we take $\tau_\alpha$ to be the character of $\mathcal{K}$ which extends $\psi_\alpha$ and which satisfies
\begin{equation}
\tau_\alpha(\beta)=(-1)^{v_E(\beta)},\;\beta\in E^\times.
\end{equation}

For $E/F$ unramified we have the following

\begin{prop}
\label{existenceoftau} There exists a unique representation $\tau_\alpha$ of $\mathcal{K}$ lying over $\psi_\alpha$ which has the property that $\tr\tau_\alpha(\zeta)=-1$ for a root of unity $\zeta\in E^\times\backslash F^\times$.
\end{prop}

\begin{proof}
This is an exercise in representation theory.  It can also be deduced from Prop.~\ref{hermitiandecomp} once one notices that $\mathcal{K}^1/\mathcal{L}^1$ is isomorphic to the group $Q$ described in~\ref{Hermitian}, while the image of $\mathcal{P}^1$ is isomorphic to the subgroup $P\subset Q$.
\end{proof}

Returning to the general case, we let $\tau_\alpha^1$ be the restriction of $\tau_\alpha$ to $\mathcal{K}^1$.

\begin{lemma}
\label{delta} Every irreducible representation of $\mathcal{K}$ lying over $\tau_\alpha^1$ is of the form $\tau_\alpha\times\omega\circ\delta$ for some character $\omega$ of $F^\times$.
\end{lemma}
\begin{proof}  Consider the map $\delta\from\mathcal{K}\to F^\times$.  The image $\delta(\mathcal{K})$ is a certain group $U_F^r$.  Explicitly, $r=(m+1)/2$ if $E/F$ is ramified and $r=m$ if $E/F$ is unramified.  There exists a section $d\from U_F^r\to \mathcal{P}$ of $\delta$, so that $\delta(d(x))=x$ for all $x\in U_F^r$.   The lemma follows formally from the fact that $\delta(U_F^r)$ normalizes the representation $\tau_\alpha^1$.
\end{proof}

Let $J_A$ and $J_B$ be as in~\ref{admissiblepairs}.
\begin{lemma} \label{normalizes} The group $J_A\times J_B$ normalizes the group $\mathcal{K}$ and the representation $\tau_\alpha$.
\end{lemma}

Let $Y_\alpha$ be the set of characters $\chi$ of $E^\times$ of the form $\chi=\chi'\times(\omega\circ N_{E/F})$, where $\chi'(1+x)=\psi_E(\alpha x)$ for all $x\in\gp_E^m$.   Thus $Y_\alpha$ is a larger set of characters than $X_\alpha$.

\begin{prop} \label{J} Let $\Lambda$ be an irreducible representation of $J_A\times J_B$.  Then $\Lambda$ lies over the representation $\tau_\alpha^1$ of $\mathcal{K}^1$ if and only if there exists a character $\chi\in Y_\alpha$ for which $\Lambda\isom\Lambda_\chi\otimes\check{\Lambda}_\chi'$.
\end{prop}

\begin{proof}
It is a simple matter to show that for all $\chi\in Y_\alpha$ we have that the restriction of the representation $\Lambda_\chi\otimes\check{\Lambda}_\chi'$ to $\mathcal{K}^1$ equals $\tau_{\alpha}^1$.  Therefore assume $\Lambda$ is an irreducible representation of $J_A\times J_B$ lying over $\tau_\alpha^1$.

By Lemma~\ref{delta}, the restriction of $\Lambda$ to $\mathcal{K}$ decomposes into representations of the form $\tau_\alpha\times(\omega\circ\delta)$.  By Lemma~\ref{normalizes}, $J_A\times J_B$ stabilizes the subspace of $\Lambda$ on which $\mathcal{P}$ acts by a particular character $\psi_\alpha\times(\omega\circ\delta)$; since $\Lambda$ is irreducible, only one such representation may appear.  By replacing $\Lambda$ with $\Lambda\otimes(\omega^{-1}\circ \delta)$ we may assume that $\omega$ is the trivial character, so that $\Lambda$ lies over $\tau_\alpha$.

Write $\Lambda=\Lambda_1\otimes\Lambda_2$, where $\Lambda_1$ and $\Lambda_2$ are irreducible representations of $J_A$ and $J_B$, respectively.   Let $\chi^1$ be a character of $U_E^1$ which appears in $\Lambda_1$, and let $W$ be the largest subspace of $\Lambda_1$ on which $U_E^1$ acts through $\chi^1$.
 If $u\in U_{\mathfrak{A}}^{\floor{m/2}+1}$, then for $w\in W$ we have  $\Lambda(u)w=\Lambda_1(us(u)^{-1})\Lambda_1(s(u))w=\chi^1(s(u))\Lambda_1(us(u)^{-1})w$.  On the other hand, $us(u)^{-1}\in 1+W_A^m$ and therefore $(us(u)^{-1},1)\in \mathcal{P}$.   Since $\Lambda$ lies over $\psi_\alpha$ we have that $\Lambda_1(us(u)^{-1})$ acts by the scalar $\psi_\alpha(us(u)^{-1},1)=\psi(\tr_{A/F} \alpha (us(u)^{-1}-1))=1$.  We find that $U_{\mathfrak{A}}^{\floor{m/2}+1}$ acts on $W$ through the character $u\mapsto \chi^1(s(u))$.  Let $\beta\in E^\times$ be such that $\chi^1(s(1+x))=\psi_A(\beta x)$ for $x\in \gP_{\mathfrak{A}}^{\floor{m/2}+1}$.  Since $J_A$ normalizes this character, $J_A$ preserves $W$, and since $\Lambda_1$ is irreducible, $\Lambda_1=W$.  Therefore $\Lambda_1$ lies over the character of $U_{\mathfrak{A}}^{\floor{m/2}+1}$ given by $x\mapsto \psi_A(\beta x)$ for $x\in\gP_{\mathfrak{A}}^{\floor{m/2}+1}$.  By a similar argument, $\Lambda_2$ lies over the character of $U_{B}^{f(m+1)/2}$ given by $1+x\mapsto\psi_B(-\beta x)$ for $x\in \gP_B^{f(m+1)/2}$.

We first consider the case where $m$ is odd.   We have that $\Lambda_1$ lies over the character $\eta$ of $U^1_EU^{(m+1)/2}_{\mathfrak{A}}$ which is $u\mapsto \chi^1(u)$ on $U^1_E$ and $1+x\mapsto \psi_A(\alpha x)$ for $1+x\in U^{(m+1)/2}_{\mathfrak{A}}$.  But $J_A=E^\times U_{\mathfrak{A}}^{(m+1)/2}$, and it is easy to see that the only representations of $J_A$ which lie over $\eta$ are precisely the characters $\Lambda_\chi$ where $\chi$ is a character of $E^\times$ extending $U^1_E$.   Thus $\Lambda_1\isom\Lambda_\chi$ for a character $\chi\in X_\alpha$.

Now we claim that $\Lambda_2\isom\check{\Lambda}'_\chi$.  If $E/F$ is ramified, then $\tau_\alpha$ was defined to restrict to $E^\times$ as the character $\beta\mapsto (-1)^{v_E(\beta)}$.  Therefore the group $j_{x,B}(E^\times)\subset J_B$ must act on $\Lambda_2$ through the character $\beta\mapsto (-1)^{v_E(\beta)}\theta(\beta^{-1})$.  We also have that $\Lambda_2$ lies over the character $1+x\mapsto \psi_B(-\beta x)$ of $U_B^{(m+1)/2}$.  By definition we have $\Lambda_2\isom\check{\Lambda}_\chi'$.

Now suppose $E/F$ is unramified.  Since $\tau_\alpha$ lies over the identity character of $U^1_E$, we find that $\Lambda_2$ lies over the character $\check{\chi}\vert_{U^1_E}$.  Now let $\zeta$ be a root of unity in $E^\times\backslash F^\times$; we have $\tr \Lambda(\zeta)=\tr \Lambda_1(\zeta)\Lambda_2(\zeta) =-1$, implying that $\tr\Lambda_2(\zeta)=-\check{\chi}(\zeta)$.  Finally, we have already seen that $\Lambda_2$ lies over the character $1+x\mapsto \psi_B(-\beta x)$ of $U_B^{m+1}$.  We conclude from the description of $\Lambda'_\chi$ in~\ref{admissiblepairs} that $\Lambda_2\isom\check{\Lambda}'_\chi$.

The argument is similar in the case of $m$ even and $E/F$ unramified, except the roles of $\Lambda_1$ and $\Lambda_2$ are reversed.
\end{proof}

\begin{prop} \label{pipiprime} Let $\pi$ and $\pi'$ be smooth irreducible representations of $\GL_2(F)$ and $B^\times$, respectively.  Then $\Hom_{\GL_2(F)\times B^\times}\left(\pi\otimes\pi',\Ind_{\mathcal{K}^1}^{\GL_2(F)\times B^\times}\tau_\alpha^1\right)$ has dimension 1 if $\pi=\pi_\chi$ and $\pi'=\check{\pi}'_\chi$ for some $\chi\in Y_\alpha$.  Otherwise, it vanishes.
\end{prop}

\begin{proof}  By Frobenius reciprocity, $$\Hom_{\GL_2(F)\times B^\times}\left(\pi\otimes\pi',\Ind_{\mathcal{K}^1_{x,m}}^{\GL_2(F)}\tau_\alpha^1\right)=\Hom_{J_A\times J_B}\left(\pi\otimes\pi'\vert_{J_A\times J_B},\Ind_{K_{x,m}^1}^{J_A\times J_B}\tau_\alpha^1\right).$$ By Prop.~\ref{J}, the dimension of this space is the number of characters $\chi\in Y_\alpha$ for which $\Lambda_\chi\otimes\check{\Lambda}_\chi$
 is contained in $\pi\otimes\pi'$.  If there exists one such character $\chi$, then already we have $\pi=\pi_\chi$ and $\pi'=\check{\pi}'_\chi$, since $\pi_\chi=\Ind_{J_A}^{\GL_2(F)}\Lambda_\chi$ and $\check{\pi}_\chi=\Ind_{J_B}^{B^\times}\check{\Lambda}_\chi$ are irreducible.  There is only one other character $\chi'$ for which $\pi=\pi_{\chi'}$, namely the $F$-conjugate character $\chi'=\chi^\sigma$.   We claim that $\chi^\sigma$ does not belong to $Y_\alpha$.  Assume it does:  then we would have  $\chi^\sigma(1+x)=\chi(1+x^\sigma)=\psi_E(\alpha x^\sigma)=\psi_E(\alpha^\sigma x)=\psi_E(\alpha x)$ for all $x\in\gp_E^m$. This implies that $\alpha-\alpha^\sigma\in\gp_E^{1-m}$, which contradicts the fact that $\alpha$ is minimal.
 \end{proof}

 As a direct consequence of Prop.~\ref{pipiprime}, we find that if $\pi$ is a smooth irreducible representation of $\GL_2(F)$, the space $\Hom_{\GL_2(F)}\left(\pi,\Ind_{\mathcal{K}_{x,m}^1}^{\GL_2(F)\times B^\times}\tau_\alpha^1\right)$ equals $\JL(\pi)$ if $\pi=\pi_\chi$ for some $\chi\in Y_\alpha$, and it vanishes otherwise.

\section{Realization of the Jacquet-Langlands correspondence}
\label{JLrealization}

Let $E/F$ be a tame quadratic extension, let $x\in M^E$ and let $m\geq 0$ be an integer.  The goal of this section is to prove the following

\begin{Theorem}  \label{existenceofX} There exists a smooth, projective, geometrically connected curve $\mathfrak{X}_{x,m}$ over $\overline{k}$, together with a $\overline{k}$-linear action of $\mathcal{K}_{x,m}^1$ on $\mathfrak{X}_{x,m}$ with the following properties:
\begin{itemize}
\item[(i)] The action $\mathcal{K}_{x,m}^1\to\Aut\mathfrak{X}_{x,m}$ has kernel exactly $\mathcal{L}_{x,m}^1$.
\item[(ii)] For all smooth irreducible representations $\pi$ of $\GL_2(F)$, we have that $$\Hom_{\GL_2(F)}\left(\pi,\Ind_{\mathcal{K}_{x,m}^1}^{\GL_2(F)\times B^\times}H^1(\mathfrak{X}_{x,m},\QQ_\ell)\right)$$ equals $\JL(\pi)^{\oplus 2}$ if $\pi=\pi_\chi$ for some character $\chi$ of $E^\times$ of essential level $m$, and vanishes otherwise.
%
\end{itemize}
\end{Theorem}

The proof will be done case by case over the next few paragraphs.  We make a few abbreviations which will apply for the remainder of the section.  Since $x\in M^E$ is given, we {\em identify} $E$ with a subfield of the algebra $M_2(F)\times B$ by means of the injection $j_x$.  We fix an integer $m\geq 0$ as well, and we write $\mathfrak{A}$, $\mathcal{K}$, $\mathfrak{X}$, etc. for the objects $\mathfrak{A}_x$, $\mathcal{K}_{x,m}$, $\mathfrak{X}_{x,m}$, etc.

We also introduce the notations $\mathcal{Q}=\mathcal{K}/\mathcal{L}$, $\mathcal{Q}^1=\mathcal{K}^1/\mathcal{L}^1$.

\subsection{Level zero supercuspidals}  In this paragraph, $E/F$ is unramified, $x\in M^E$, and $m=0$.  By replacing $x$ by one of its $\GL_2(F)$-translates we may assume $\mathcal{A}=M_2(\OO_F)$.  Then $\mathcal{K}=F^\times(\GL_2(\OO_F)\times\OO_B^\times)$.

Let $\theta$ be a character of $k_E^\times$ for which $\theta\neq\theta^q$.  Let $\lambda_\theta$ be the cuspidal representation of $\GL_2(k_E)$ corresponding to $\theta$.  Let $\tau_\theta$ be the representation of $\mathcal{K}$ which is trivial on $F^\times$ and for which $\tau_\theta\vert_{\GL_2(\OO_F)\times\OO_B^\times}$ is the inflation of $\lambda_\theta\otimes\check{\theta}$.  Let $\tau_\theta^1$ be the restriction of $\tau_\theta$ to $\mathcal{K}^1$

\begin{lemma} Every irreducible representation of $\mathcal{K}$ lying over the representation $\tau_\theta^1$ of $\mathcal{K}^1$ is of the form $\tau_\theta\otimes(\omega\circ\delta)$ for some character $\omega$ of $F^\times$.
\end{lemma}

\begin{prop} \label{tautheta} Let $\pi$ be a smooth irreducible representation of $\GL_2(F)$.  Then $$\Hom_{\GL_2(F)}\left(\pi,\Ind_{\mathcal{K}^1}^{\GL_2(F)\times B^\times}\tau_\theta^1\right)$$ equals $\JL(\pi)$ if $\pi=\pi_\chi$ for some character $\chi\in X_\theta$, and equals 0 otherwise.
\end{prop}

Now let $\mathfrak{X}$ be the Deligne-Lusztig curve for the group $\SL_2(k)$.  This is the smooth projective curve with affine equation $X^qY-XY^q=1$.  The group $\mathcal{Q}^1$ is the set of pairs $(g,b)$ with $g\in \GL_2(k)$ and $\beta\in k_E^\times$ satisfying $\det g=N_{k_2/k}\beta$.  We define a (right) action of $\mathcal{Q}^1$ on $\mathfrak{X}$ via the rule
$$(u,v)^{(g,\beta)}=(\beta^{-1}(au+cv),\beta^{-1}(bu+dv))$$ when $g=\tbt{a}{b}{c}{d}$.

\begin{prop} As a module for the action of $\mathcal{Q}^1$, we have $$H^1(\mathfrak{X},\QQ_\ell)=\bigoplus_{\theta\in S}\lambda_\theta\otimes\check{\theta},$$ where $S$ is a set of representatives for the equivalence classes of characters $\theta$ of $k_E^\times$ not factoring through $N_{k_E/k}$ modulo the relation $\theta\sim \theta'$ if $\theta'\theta^{-1}$ factors through $N_{k_E/k}$.
\end{prop}
\begin{proof} This is an application of Deligne-Lusztig theory, or else an easy exercise using the Lefshetz fixed-point formula with the explicit equation for $\mathfrak{X}$.
\end{proof}

\subsection{Case of $E/F$ tamely ramified, $m$ odd}  Fix a uniformizer $\varpi_E$ for $E$.

Define a map $\rho\from \mathcal{K}^1\to \mathbf{F}_q\times \Z/2\Z$ by $j_x(\beta)\mapsto (0,v_E(\beta)\mod 2)$ and $1+t\mapsto (\varpi_E^{-m}s(t)\mod \gp_E,0)$ for $t\in V^m$.  Then $\rho$ descends to an isomorphism $\mathcal{Q}^1\to \mathbf{F}_q\times \Z/2\Z$.  Let $\mathfrak{X}$ be the smooth projective curve with affine equation $$X^q-X=Y^2,$$ and have $\mathbf{F}_q\times\Z/2\Z$ act on this curve in the following manner:  An element $a\in\mathbf{F}_q$ acts via $(X,Y)\mapsto (X+a,Y)$, and the nontrivial element of $\Z/2\Z$ acts via $(X,Y)\mapsto (X,-Y)$.

By Prop.~\ref{hyperellipticdecomp}, we have an isomorphism of $\mathcal{K}^1$ modules $$H^1(\mathfrak{X},\QQ_\ell)\isom \bigoplus_{\alpha\in S} \tau_\alpha^1,$$ where $S$ is a set of representatives for the nontrivial elements in $\gp_E^{-m}/\gp_E^{1-m}$.  Every character $\chi$ of essential level $m$ belongs to exactly one $Y_\alpha$ for some unique $\alpha\in S$.  Since $\pi_\chi=\pi_{\chi'}$ if and only if $\chi'=\chi$ or $\chi'=\chi^\sigma$, we find that $$\Hom_{\GL_2(F)}\left(\pi,\Ind_{\mathcal{K}^1}^{\GL_2(F)\times B^\times} H^1(\mathfrak{X},\QQ_\ell)\right)$$ equals $\JL(\pi)^{\oplus 2}$ if $\pi=\pi_\chi$ for some $\chi$ of essential level $m$, and vanishes otherwise.

\subsection{Case of level $m>0$, $E/F$ unramified} Here $\mathcal{Q}$ is isomorphic to the subgroup of $\PGL_3(k_E)$ consisting of matrices of the form $$\left(\begin{matrix} \alpha & \beta & \gamma \\ & \alpha^q & \beta^q \\ & & \alpha \end{matrix}\right)$$.  For an explicit isomorphism, see~\cite{Weinstein}.    The subgroup $\mathcal{Q}^1$ consists of those matrices as above which satisfy
\begin{equation}
\label{unitary}
\alpha\gamma^q+\alpha^q\gamma=\beta^{q+1}.
\end{equation}
In fact $\mathcal{Q}$ is a Borel subgroup of a unitary group in three variables associated to the quadratic extension $k_2/k$.  The image of $\mathcal{P}\subset \mathcal{K}$ in $\mathcal{Q}$ is the central subgroup
$$\set{\left(\begin{matrix} 1 & 0 & \alpha \\ & 1 & 0 \\ & & 1\end{matrix}\right)}\isom k_E.$$

Considered as a subgroup of $\PGL_3(\overline{k})=\Aut \pr^2_{\overline{k}}$, the group $\mathcal{Q}^1$ preserves the curve with projective equation $$XZ^q+X^qZ=Y^{q+1},$$ which we take for our curve $\mathfrak{X}$.

A minimal element $\alpha\in E^\times$ of valuation $-n$ gives rise to a character $\psi_\alpha$ of $\mathcal{P}$ which factors through a character of $k_E$;  the condition that $\alpha$ is minimal implies that $\psi_\alpha$ does not factor through $\tr_{k_E/k}$.   Let $\tau_\alpha$ be the representation of $\mathcal{K}$ which lies over $\psi_\alpha$ and satisfies $\tr\tau_\alpha(\zeta)=-1$ for each root of unity $\zeta\in E^\times\backslash F^\times$, as in Lemma~\ref{existenceoftau}, and let $\tau_\alpha^1$ be the restriction of $\tau_\alpha$ to $\mathcal{K}$.

Let $S$ be a set of representatives for the minimal elements of $\gp_E^{-m}/(\gp_F^{-m}+\gp_E^{1-m})$.  Then every character of $E^\times$ of essential level $m$ belongs to exactly one $Y_\alpha$ for some unique $\alpha\in S$.  Then by Prop.~\ref{hermitiandecomp} we have an isomorphism of $\mathcal{K}^1$-modules
$$H^1(\mathfrak{X},\QQ_\ell)=\bigoplus_{\alpha} \tau_\alpha^1.$$
We find that $\mathfrak{X}$ satisfies the hypotheses of Theorem~\ref{existenceofX} by proceeding as in the previous paragraph.

\subsection{Case of level $m=0$, $E/F$ ramified} \label{ramifiedlevelzero} There are no minimal elements of $E^\times$ of valuation zero, so in order to satisfy the demands of the theorem we must take $\mathfrak{X}=\pr^1$ to be the rational curve over $\overline{k}$.  By replacing $x$ by a $\GL_2(F)$-translate we may assume $\mathfrak{A}=\tbt{\OO_F}{\OO_F}{\gp_F}{\OO_F}$.  Then $\mathfrak{A}/\gP_\mathfrak{A}\isom k\times k$.

Let $k_2\subset \overline{k}$ be the quadratic extension of $k$.  We choose a $k$-isomorphism $k_2\isom k_B$.

\begin{lemma} The group $\mathcal{Q}^1$ is isomorphic to the semidirect product $k_2^\times\rtimes (\Z/2\Z)$, where the nontrivial element of $(\Z/2\Z)$ acts as $\beta\mapsto \beta^{-1}$ on $k_2^\times$.
\end{lemma}
\begin{proof}  We have
$$
\mathcal{K}=E^\times\left(\tbt{\OO_F^\times}{\OO_F}{\gp_F}{\OO_F}\times \OO_B^\times\right).$$
Define a homomorphism $\rho\from\mathcal{K}^1\to k_2\rtimes (\Z/2\Z)$ by
\begin{eqnalign*}
\rho(\alpha)&=(0,v_E(\alpha)\text{ mod } 2)\\
\rho\left(g,\beta\right)&=(a^{-1}\beta\text{ mod }\gP_B,0)
\end{eqnalign*}
for every $\alpha\in E^\times$ and every $(g,b)\in\mathfrak{A}^\times\times\OO_B^\times$ satisfying $\det g=N_{B/F}b$;  here $a$ denotes the upper left entry of $g$.  Then $\rho$ descends to an isomorphism $\mathcal{Q}^1\isom k_2\rtimes \Z/2\Z$.
\end{proof}

We give a faithful action of $k_2^\times \rtimes\Z/2\Z$ on $\pr^1$:    The group $k_2^\times$ acts by multiplication ($X\mapsto \beta X$) and $\rho(\varpi_E)$ acts by inversion ($X\mapsto 1/X$).

\subsection{Case of level $m>0$ even, $E/F$ ramified} \label{ramifiedevenlevel}
Once more, there are no minimal elements of even valuation, so we take $\mathfrak{X}=\pr^1$.

\begin{lemma} The group $\mathcal{Q}^1$ is (noncanonically) isomorphic to $k_2\rtimes (\Z/2\Z)$, where the nontrivial element of $(\Z/2\Z)$ acts as $\beta\mapsto -\beta$ on $k_2^\times$.
\end{lemma}
\begin{proof}
By Lemma~\ref{GR1}[(iii)] and Lemma~\ref{Kxm}[(iii)], $\mathcal{Q}^1$ is the semidirect product of the 2-dimensional $k$-vector space $V^m/W^m$ by the group $E^\times/F^\times U_E^1$, which has order 2.   Write $\mathcal{A}=\OO_E\oplus\mathfrak{C}_A$, where $\mathfrak{C}$ is the orthogonal complement to $\OO_E$ under the trace map.  Similarly write $\OO_B=\OO_E\oplus\mathfrak{C}_B$.  Let $\mathfrak{C}=\mathfrak{C}_A\times\mathfrak{C}_B$.  We have $V^m/W^m\isom \gp_E^{m/2}\mathfrak{C}/\gp_E^{m/2+1}\mathfrak{C}$.  The lemma follows once we observe that conjugation by a uniformizer $\pi_E\in E$ acts as negation on $\gp_E^{m/2}\mathfrak{C}/\gp_E^{m/2+1}\mathfrak{C}$.  
\end{proof}

We give a faithful action of $k_2\rtimes \Z/2\Z$ on $\pr^1$: The group $k_2$ acts by translation ($X\mapsto X+\beta$) and $\rho(\varpi_E)$ acts by reflection about the origin ($X\mapsto -X$).

\section{Realization of the local Langlands Correspondence}
\label{LLCrealization}

Let $(E/F,\chi)$ be an admissible pair, so that $\Ind_{E/F}\chi$ is an irreducible representation of $W_F$.  It is not the case that $\Ind_{E/F}\chi\mapsto\pi_\chi$ is the local Langlands correspondence.  For instance, the central character of $\pi_\chi$ is $\chi\vert_F$, while the central character of $\Ind_{E/F}\chi$ is $\chi\vert_F\cdot\kappa_{E/F}$, where $\kappa_{E/F}\from F^\times\to\set{\pm 1}$ is the character which cuts out the quadratic extension $E/F$ by local class field theory.

To remedy this situation, this ``na\"ive" correspondence must be adjusted by replacing $\chi$ with the product $\chi\Delta_\chi$, where $\Delta_\chi$ is an appropriately chosen tamely ramified character of $E^\times$.  We describe a character $\Delta_\chi$ so that $\Ind_{E/F}\chi\Delta_\chi\mapsto\pi_\chi$ is the $\ell$-adic local Langlands correspondence.  If $E/F$ is unramified, then $\Delta_\chi$ is the unramified character with $\Delta_\chi(\varpi_E)=-q$.

If $E/F$ is ramified, the definition of $\Delta_\chi$ is more involved.  Let $\psi$ be a $\QQ_\ell$-valued character of $F$ which vanishes on $\gp_F$ but not on $\OO_E$.  Define the Gauss sum
$$\tau(\kappa_{E/F},\psi)=\sum_{a\in \OO_F/\gp_F}\kappa_{E/F}(a)\psi(a),$$
so that $\tau(\kappa_{E/F},\psi)^2=\kappa_{E/F}(-1)q$.   Suppose $\alpha\in\gp_E^{-m}$ is a minimal element for which $\chi\in Y_m$.  Then $\Delta_\chi$ is the character satisfying
\begin{itemize}
\item[(i)] $\Delta_\chi$ vanishes on $U_E^1$.
\item[(ii)] For $u \in U_E^1$, we have $\Delta_\chi(\alpha)=\left(\frac{u\mod \gp_E}{q}\right)$.
\item[(iii)] $\Delta(\varpi_E)=\kappa_{E/F}(\zeta)\tau(\kappa_{E/F},\psi)^mq^{(1-m)/2}$, where $\zeta\in \OO_F^\times$ satisfies $\alpha\varpi_E^m\equiv \zeta\pmod{\gp_E}$.
\end{itemize}


\begin{Theorem} \label{elladic} $\Ind_{E/F}\chi\Delta_\chi\mapsto\pi_\chi$ is the $\ell$-adic local Langlands correspondence.
\end{Theorem}

This is essentially Theorem 34.4 in~\cite{Henniart:Bushnell}, adapted to the situation of the $\ell$-adic local Langlands correspondence.  In terms of characters with complex coefficients, this means replacing the $\Delta_\chi$ of ~\cite{Henniart:Bushnell} with the character $x\mapsto \Delta_\chi(x)\abs{x}_E^{f/2}$, where $f$ is the residual degree of $E/F$.



\subsection{Semilinear Actions.}  Let $x\in M^E$ and let $m\geq 0$.  Theorem~\ref{existenceofX} gives a subgroup $\mathcal{K}^1\subset\GL_2(F)\times B^\times$ which acts linearly on a curve $\mathfrak{X}$ through a quotient $\mathcal{Q}^1$ in such a way that the the cohomology of the fiber product $\mathfrak{X}\times_{\mathcal{K}^1} (\GL_2(F)\times B^\times)$ realizes the Jacquet-Langlands correspondence for supercuspidal representations of the form $\pi_\chi$, where $\chi$ is an admissible character of $E^\times$ of essential level $m$.  We make the abbreviation $$\Ind\mathfrak{X}=\mathfrak{X}\times_{\mathcal{K}^1}(\GL_2(F)\times B^\times).$$

We wish to define a {\em semilinear} action of $W_F$ on $\Ind\mathfrak{X}$ in such a way that the cohomology of the fiber product realizes the $\ell$-adic local Langlands correspondence as well.

Let $\mathcal{N}=\mathcal{N}_x\subset \GL_2(F)\times B^\times$ be the group from~\ref{relativeweil}.   Observe that $\mathcal{N}$ normalizes $\mathcal{K}^1$.  We will be interested in actions $$\Delta\from \mathcal{N}\to \Aut\mathfrak{X}$$ which satisfy the requirements
\begin{itemize}
\item[(i)] $\Delta$ is semilinear with respect to the homomorphism $\mathcal{N}\to\Gal(\overline{k}/k)$, $t\mapsto \Fr_q^{v_F(\delta(t))}$ (in the sense of Defn.~\ref{semilineardefn}).
\item[(ii)] $\Delta$ extends to a well-defined action of $\mathcal{K}^1\rtimes \mathcal{N}$ which agrees with the action of $\mathcal{K}^1$ from Theorem~\ref{existenceofX}.
\item[(iii)] For $\beta\in W_E^{\text{ab}}=E^\times$, we have that the action of $j_x(\beta)\in \mathcal{K}^1$ on $\mathfrak{X}$ agrees with $\Delta(j_{x,W}(\beta))$.
\end{itemize}

Explicitly, the second condition means that the equation
\begin{equation}
\label{semidirect}
\Delta(n)^{-1}k\Delta(n)=n^{-1}kn
\end{equation}
holds in $\Aut \mathfrak{X}$
for all $n\in \mathcal{N}$, $k\in \mathcal{K}^1$.
Given such an action, we may define an action of $W_F$ on $\Ind\mathfrak{X}$ as follows.  A point of $\Ind\mathfrak{X}$ may be represented by a pair $(P,t)$, where $P\in\mathfrak{X}$ and $t\in\GL_2(F)\times B^\times$.  Let $w\in W_{E/F}$, and let $\tilde{w}\in\mathcal{N}$ be a lift of $j_{x,W}(w)\in j_x(E^\times)\backslash \mathcal{N}$.  Then we define, for $w\in W_F$: $$(P,t)^w=(P^{\Delta(\tilde{w})},\tilde{w}^{-1}t).$$  Then the condition (iii) above shows that this definition does not depend on the choice of lift $\tilde{w}$, and the condition in Eq.~\ref{semidirect} means exactly that this action preserves the equivalence relation $(P^k,t)\sim (P,kt)$ for $k\in\mathcal{K}^1$.  We arrive at a well-defined semilinear action
\begin{equation}
\GL_2(F)\times B^\times\times W_F\to\Aut\Ind\mathfrak{X}.
\end{equation}
This action induces an action of the same triple product group on $\Ind_{\mathcal{K}^1}^{\GL_2(F)\times B^\times} H^1(\mathfrak{X},\QQ_\ell)$.

Call an action $\mathcal{N}\to\Aut\mathfrak{X}$ {\em compatible with the action of $\mathcal{K}^1$} if it satisfies (i)--(iii) above.


Recall that each admissible character $\chi$ of $E^\times$ of essential level $m$ determines representations $\Lambda_\chi$ and $\Lambda_{\chi'}$ of $\GL_2(F)$ and $B^\times$ respectively for which $\pi_\chi=\Ind_{J}^{\GL_2(F)}\Lambda_\chi$ and $\pi'_{\chi}=\Ind_{J'}\Lambda'_{\chi}$.  The tensor product representation $\Lambda_\chi\otimes\Lambda'_{\chi}$ is an extension to $J_A\times J_B$ of an irreducible representation of $\mathcal{K}^1$ of the form $\tau_\alpha^1$, where $\alpha\in E^\times$ is a minimal element of valuation $-m$.   We have that
$$\Hom_{\mathcal{K}^1}\left(\Lambda_\chi\otimes\Lambda_{\check{\chi}}'\vert_{\mathcal{K}^1},H^1(\mathfrak{X},\QQ_\ell)\right)
=\Hom_{\mathcal{K}^1}\left(\tau_\alpha,H^1(\mathfrak{X},\QQ_\ell)\right)$$
is one-dimensional.

Now suppose $\Delta$ is an action of $\mathcal{N}$ on $\mathfrak{X}$ which is compatible with the action of $\mathcal{K}^1$.  Then there is a well-defined action of $E^\times$ on $\Hom_{\mathcal{K}^1}\left(\Lambda_\chi\otimes\Lambda_{\check{\chi}}'\vert_{\mathcal{K}^1},H^1(\mathfrak{X},\QQ_\ell)\right)$ defined as follows.  If $\lambda$ is a $\mathcal{K}^1$-equivariant map from $\Lambda_\chi\otimes\Lambda_{\check{\chi}}'$ into $H^1(\mathfrak{X},\QQ_\ell)$, and $\beta\in \mathcal{N}$, then define $\lambda^\beta$ to be the map $v\mapsto \Delta(\beta)\lambda(j(\beta,1)^{-1}v)$.   The linear map $\lambda^\beta$ so defined is $\mathcal{K}^1$-equivariant precisely because $\Delta$ is compatible with the action of $\mathcal{K}^1$.



\begin{Theorem} \label{Weilaction} There exists a semilinear action $\Delta=\Delta_{x,m}\from W_F\to\Aut\mathfrak{X}$ compatible with the action of $\mathcal{K}^1$ with the following property:  For every admissible character $\chi$ of $E^\times$ of essential level $m$, the group $E^\times$ acts on $\Hom_{\mathcal{K}^1}\left(\Lambda_\chi\otimes\Lambda_{\check{\chi}}'\vert_{\mathcal{K}^1},H^1(\mathfrak{X},\QQ_\ell)\right)$ through the character $\check{\chi}\Delta_{\check{\chi}}$.
\end{Theorem}

The remainder of the section will be devoted to proving Theorem~\ref{Weilaction}.   Assume it now, and let $H^1(\Ind\mathfrak{X},\QQ_\ell)$ be given the action of $W_F$ arising from the action $\Delta$ of the Theorem.  We record the following consequence:

\begin{Cor} \label{realizesLLC} Let $\pi$ be a smooth irreducible representation of $\GL_2(F)$.  We have that $\Hom_{\GL_2(F)}\left(\pi,H^1(\Ind \mathfrak{X},\overline{\Q}_\ell)\right)$ equals $\JL(\pi)\otimes\mathcal{L}_\ell(\check{\pi})$ when $\pi=\pi_\chi$ for an admissible character $\chi$ of essential level $m$, and vanishes otherwise.
\end{Cor}

\begin{proof}   By Theorem~\ref{existenceofX}, the space $\Hom_{\GL_2(F)}\left(\pi,H^1(\Ind \mathfrak{X},\overline{\Q}_\ell)\right)$ equals $0$ if $\pi$ is not of the form $\pi_\chi$ for $\chi$ an admissible character of level $m$.  Therefore suppose $\pi=\pi_\chi$ for such a character.  Then again by Theorem~\ref{existenceofX} we have that $$\rho=\Hom_{\GL_2(F)\times B^\times}\left(\pi\otimes\JL(\check{\pi}),H^1(\Ind\mathfrak{X},\QQ_\ell)\right)$$ is a 2-dimensional representation of $W_{E/F}$.   We claim that $\rho=\mathcal{L}_\ell(\check{\pi})$.  By Theorem~\ref{elladic}, $\mathcal{L}_\ell(\pi)=\Ind_{E/F}\chi\Delta_\chi$.  Therefore to prove Cor.~\ref{realizesLLC} it is enough to show that $\rho\vert_{E^\times}$ contains the character $\check{\chi}\Delta_{\check{\chi}}$.

We have that $\pi\otimes\JL(\check{\pi})=\Ind_{J_A\times J_B}^{\GL_2(F)\times B^\times} \Lambda_\chi\otimes\Lambda'_{\check{\chi}}$.
By Frobenius reciprocity, $$\rho=\Hom_{\mathcal{K}^1}\left(\left(\Ind_{J_A\times J_B}^{\GL_2(F)\times B^\times} \Lambda_\chi\otimes\Lambda'_{\check{\chi}}\right)\vert_{\mathcal{K}^1},H^1(\mathfrak{X},\QQ_\ell)\right).$$  By Mackey's theorem, $\rho\vert_{E^\times}$ contains $\Hom_{\mathcal{K}^1}\left( \left(\Lambda_\chi\otimes\Lambda'_{\check{\chi}}\right)\vert_{\mathcal{K}^1},H^1(\mathfrak{X},\QQ_\ell)\right)$, which equals the character $\check{\chi}\Delta_{\check{\chi}}$ by Theorem~\ref{Weilaction}.
\end{proof}

\subsection{Case of $m=0$, $E/F$ unramified}
\label{levelzerocalc}

In this case, $\mathfrak{X}$ is the nonsingular projective curve with affine equation $XY^q-X^qY=1$.   Let $\Phi=(\Phi_1,\Phi_2)\in\mathcal{N}$ be a lift of the nontrivial element of $\Gal(E/F)$ for which $\Phi_1\in\mathfrak{A}^\times$, $\Phi_2\in\gP_B^{-1}$.  Then the group  $\mathcal{N}$ is generated by $j_x(E^\times)$, the subgroup $\set{1}\times j_{x,B}(E^\times)$ and the element $\Phi$.

We now describe the required action $\Delta\from \mathcal{N}\to\Aut\mathfrak{X}$.  The automorphism $\Delta(\Phi)$ will be the one lying over $\Fr_q\in\Gal(\overline{k}/k)$ which effects $(X,Y)\mapsto (uX,uY)$ on coordinates;  here $u\in\overline{k}$ is a root of $u^{q+1}=-1$.   Now let $\beta\in E^\times$;  the automorphism $\Delta((1,j_{x,B}(\beta))$ will be the one lying over $\Fr_q^{-2v_E(\beta)}$ which is trivial on coordinates.  Finally, the action of $j_x(E^\times)$ through $\Delta$ shall be as demanded by condition (3) of the previous paragraph.   It is easily checked that this defines a well-defined action $\Delta\from\mathcal{N}\to\Aut\mathfrak{X}$ which is compatible with the action of $\mathcal{K}^1$.





By Prop.~\ref{hermitian}, $\Delta(j_x(\varpi_F))$ acts on $H^1(\mathfrak{X},\QQ_\ell)$ as the scalar $-q$.    Let $\chi$ be an admissible character of $E^\times$ of essential level 0.    Thus if $\beta\in E^\times$, the action of $(1,j_{x,B}(\beta^{-1}))$ on $H^1(\mathfrak{X},\QQ_\ell)$ is through $\Delta_\chi(\beta)$.  On the other hand, $\Lambda'_{\chi}$ is a character of $F^\times\OO_B^\times$ for which $\Lambda'_\chi(j_{x,B}(\beta))=\chi(\beta)$.  Therefore the action of $E^\times\subset W_{E/F}$ on $\Hom_{\mathcal{K}^1}\left(\Lambda_\chi\otimes\Lambda'_{\check{\chi}}\vert_{\mathcal{K}^1},H^1(\mathfrak{X},\QQ_\ell)\right)$ is through the character $\check{\chi}\Delta_{\check{\chi}}$.  This establishes Thm.~\ref{Weilaction} in this case.

\subsection{Case of $m>0$, $E/F$ unramified} Here $\mathfrak{X}$ is the projective curve with affine equation $X^q+X=Y^{q+1}$.   Assume for the moment that $m$ is odd.  Let $\Phi\in\mathcal{N}$ be as in the previous paragraph, and let $\beta\in  E^\times$.  There is a unique semilinear action $\Delta\from\mathcal{N}\to\Aut\mathfrak{X}$ compatible with the action of $\mathcal{K}^1$ for which $\Delta(\Phi)$ and $\Delta(j_{x,A}(\beta),1)$ are trivial on coordinates.  Let $\chi$ be an admissible character of $E^\times$ of essential level $m$.  Then $\Lambda_\chi$ is a character whose restriction to $j_{x,A}(E^\times)$ is simply $\chi$.  By an argument rather similar to that of the previous paragraph, $E^\times\subset W_{E/F}$ acts on $\Hom_{\mathcal{K}^1}\left(\Lambda_\chi\otimes\Lambda'_{\check{\chi}}\vert_{\mathcal{K}^1},H^1(\mathfrak{X},\QQ_\ell)\right)$ through the character $\check{\chi}\Delta_{\check{\chi}}$.

The argument for $m$ even is similar, but with the roles of $A$ and $B$ reversed.


\subsection{Case of $m=0$, $E/F$ ramified}   Let $\Phi=(\Phi_1,\Phi_2)\in\mathcal{N}$ be a lift of the nontrivial element of $\Fr_q\in\Gal(\overline{k}/k)$ for which $\Phi_1\in\mathfrak{A}^\times$, $\Phi_2\in\OO_B^\times$.  Then the group $\mathcal{N}$ is generated by the subgroup $j_x(E^\times)$, the subgroup $j_{x,A}(E^\times)\times\set{1}$ and the element $\Phi$.

Here $\mathfrak{X}$ is the projective line.  Since $H^1(\mathfrak{X},\QQ_\ell)$ is trivial, it is enough to find an action $\Delta$ of $\mathcal{N}$ on $\mathfrak{X}$ which is compatible with the action of $\mathcal{K}^1$.  For $\beta\in E^\times$ we set $\Delta((j_{x,A}(\beta),1)$ to be the automorphism which is trivial on coordinates, whereas $\Delta(\Phi)$ will be the inversion map $X\mapsto 1/X$.


\subsection{Case of $m>0$ even, $E/F$ ramified}  Once again, $\mathfrak{X}$ is the projective line.  We set $\Delta(j_{x,A}(\beta),1)$ to be trivial on coordinates, while $\Delta(\Phi)$ will be the linear map defined by $X\mapsto (-1)^{m/2}X$.
Then $\Delta$ is compatible with the action of $\mathcal{K}^1$.

\subsection{Case of $m>0$ odd, $E/F$ ramified}  Here $\mathfrak{X}$ is the projective curve with affine equation $X^q-X=Y^2$.   The group $\mathcal{Q}^1$ is the direct product of the group $\gp_E^m/\gp_E^{m+1}$ by a group of order 2.  Choose a uniformizer $\varpi_E$ of $E$;  we get an isomorphism $\gp_E^m/\gp_E^{m+1}\to k$ by $t\mapsto t\varpi_E^{-m}\pmod{\gp_E}$.  Assume an element $t=1+\varpi_E^mr\in \gp_E^m/\gp_E^{m+1}\subset\mathcal{Q}^1$ acts on $\mathfrak{X}$ through $(X,Y)\mapsto (X+r,Y)$.

We define a semilinear action $\Delta$ of $\mathcal{N}$ on $\mathfrak{X}$ as follows:
\begin{eqnarray*}
\Delta((j_{x,A}(a),1))(X,Y)&=&\left(X,\left(\frac{a}{q}\right)Y\right),\;a\in\OO_E^\times\\
\Delta((j_{x,A}(\varpi_E),1))(X,Y)&=&(X,\eps Y)\\
\Delta(\Phi)(X,Y)&=&(-X,\sqrt{-1}Y)
\end{eqnarray*}
Here $\eps=-(-1)^{\frac{q-1}{2}\frac{m-1}{2}}$.

Now let $\alpha\in E^\times$ be minimal of valuation $-m$, and let $\chi\in Y_\alpha$.   Let $\tau_\alpha^1$ be the character of $\mathcal{K}^1$ as defined in~\ref{charactersofK}.
By Prop.~\ref{frobeigenvalue} we have
\begin{eqnarray*}
\tr\left(\Delta((j_{x,A}(\varpi_E),1))\biggm\vert H^1(\mathfrak{X},\QQ_\ell)^{\tau_\alpha^1}\right) &=& -\eps \sum_{t\in \gp_E^m/\gp_E^{m+1}} \left(\frac{t\beta\;\text{mod}\;\gp_E}{q}\right)\nu(t) \\
&=& -\eps\sum_{r\in\OO_E/\gp_E} \left( \frac{r\alpha^{-1}\varpi_E^n\;\text{mod}\;{\gp_E}}{q}\right)\psi(r)\\
&=& \kappa_{E/F}(-1)^{(m-1)/2}\kappa_{E/F}(\zeta)\tau(\kappa_{E/F},\psi),
\end{eqnarray*}
where $\zeta\in\OO_F^\times$ satisfies $\alpha\varpi_E^m\equiv\zeta\pmod{\gp_E}$.   We claim this equals $\Delta_\chi(\varpi_E)$.  Indeed,
\begin{eqnarray*}
\Delta_\chi(\varpi_E)&=&\kappa_{E/F}(\zeta)\tau(\kappa_{E/F},\psi)^mq^{(1-m)/2}\\
&=&\kappa_{E/F}(\zeta)\kappa_{E/F}(-1)^{(m-1)/2}\tau(\kappa_{E/F},\psi)
\end{eqnarray*}
because $\tau(\kappa_{E/F},\psi)^2=\kappa_{E/F}(-1)q$.

Meanwhile, the character $\Lambda_\chi$ takes the value $\chi(\alpha)$ on an element $j_{x,A}(\alpha)$.  We find that $E^\times\subset W_{E/F}$ acts on the space $\Hom_{\mathcal{K}^1}\left(\Lambda_\chi\otimes\Lambda'_{\check{\chi}}\vert_{\mathcal{K}^1},H^1(\mathfrak{X},\QQ_\ell)\right)$ through the character $\check{\chi}\Delta_{\check{\chi}}$.

\section{Construction of the stable Lubin-Tate curve}
\label{construction}

We now assume that $F$ has odd residual characteristic.  Thus there are three distinct quadratic extensions of $F$.  Call these $E_0$, $E_1$ and $E_2$, with $E_0/F$ unramified.  We turn to the task of gluing together the curves $\mathfrak{X}_{x,m}$ for $x\in M^{\text{CM}}$ and $m\geq 0$ to form a stable curve which has the desired properties of Thm.~\ref{mainthm}.

\subsection{Base Points}  We identify certain special points of $\mathfrak{X}_{x,m}$, which will later be used as sites of gluing.

\label{basepoints}
\begin{prop}  There exists a point $P\in \mathfrak{X}_{x,m}(\overline{k})$ with the following properties.
\begin{itemize}
\item[(i)] The stabilizer of $P$ in $\mathcal{K}_{x,m}^1$ is exactly $\mathcal{K}_{x,m+1}^1$.
\item[(ii)] There are no nontrivial $\mathcal{K}_{x,m}^1$-equivariant linear automorphisms of the pair $(\mathfrak{X}_{x,m},P)$.
\end{itemize}
Furthermore, if $Q\in\mathfrak{X}_{x,m}(\overline{k})$ is another point satisfying (i) and (ii) then there exists a $\mathcal{K}_{x,m}^1$-equivariant automorphism of $\mathfrak{X}_{x,m}$ carrying $P$ onto $Q$.
\end{prop}

\begin{proof}

In the case where $E/F$ is unramified and $m=0$, $\mathfrak{X}_{x,m}$ is the curve $XY^q-X^qY=1$.   Assume that $\mathfrak{A}_x=M_2(\OO_F)$.   We have the embeddings $j_{x,A}\from\OO_E\to M_2(\OO_F)$ and $j_{x,B}\from \OO_E\to \OO_B$.  Reducing modulo $\gp_E$ gives embeddings $k_E\injects M_2(k)$ and $k_E\injects k_B$.   Composing one map with the inverse of the other gives an embedding $\kappa\from k_B\injects M_2(k)$.  Let $P=(u,v)\in \mathbf{A}^2(\overline{k})$ be a point which satisfies $\beta P=\kappa(\beta)P$ for all $\beta\in k_B^\times$.  Let $\gamma\in\overline{k}^\times$ satisfy $\beta^{q+1}=uv^q-u^qv$.  Then $P=(\gamma^{-1}u,\gamma^{-1}v)\in\mathfrak{X}_{x,0}(\overline{k})$;  then the stabilizer of this point in $\mathcal{K}_{x,0}^1$ is $\mathcal{K}_{x,1}^1$, establishing (i).  For (ii), the only automorphisms of $\mathfrak{X}_{x,0}$ commuting with $\mathcal{K}_{x,0}^1$ are of the form $(X,Y)\mapsto (\zeta X, \zeta Y)$ with $\zeta^{q+1}=1$, and no such nontrivial automorphism could fix the point $P$.  Finally, all other points $Q\in \mathfrak{X}_{x,0}$ satisfying (i) and (ii) are of the form $\zeta P$ for $\zeta^{q+1}=1$, and this is indeed the translate of $P$ by a $\mathcal{K}_{x,0}^1$-equivariant automorphism of $\mathfrak{X}_{x,0}$.

In the case where $E/F$ is unramified and $m>0$, $\mathfrak{X}_{x,m}$ is the curve $X+X^q=Y^{q+1}$.  The point $P=(0,0)$ is the only point satisfying the required properties.

In the case where $E/F$ is ramified and $m=0$, $\mathfrak{X}_{x,0}$ is the projective line, with the action of $\mathcal{K}_{x,0}^1$ as in~\ref{ramifiedlevelzero}.   Let $P$ be the point $1\in\mathbf{P}^1(\overline{k})$;  the stabilizer in $\mathcal{K}_{x,0}^1$ of $P$ is $\mathcal{K}_{x,1}^1$.   The automorphisms of the projective line commute with the action of $\mathcal{K}_{x,0}^1$ are only $X\mapsto \pm X$, so property (2) holds.    Conversely, the only other point satisfying (i) and (ii)
is $-1$.

In the case where $E/F$ is ramified and $m>0$ is even, $\mathfrak{X}_{x,m}$ is the projective line, with the action of $\mathcal{K}_{x,0}^1$ as in~\ref{ramifiedevenlevel}.  Then $P$ is the point $0\in\mathbf{P}^1(\overline{k})$;  this is the only point satisfying (i) and (ii).

Finally, in the case where $E/F$ is ramified and $m>0$ is odd, $\mathfrak{X}_{x,m}$ is the curve $Y^2=X^q-X$.  Then $P=(0,0)$ satisfies (i) and (ii).  So does any point of the form $(a,0)$ with $a\in k$, but then this point is the translate of $P$ by the $\mathcal{K}_{x,m}^1$-equivariant isomorphism $(X,Y)\mapsto (X+a,Y)$.
\end{proof}

For every $x\in M^{\text{CM}}$ and every $m\geq 0$, we now {\em choose} a point $P_{x,m}\in\mathfrak{X}_{x,m}(\overline{k})$ satisfying the properties of Prop.~\ref{basepoints}.



\subsection{Similarity Classes of CM points.}

The vertices of the dual graph $\Gamma$ will be classes of CM points under a certain family of equivalence relations $\sim_m$ indexed by nonnegative integers.  These equivalence relations become stronger as $m$ increases.


\begin{defn}
\label{equivalencedefn}
Suppose $x,y\in M^{\text{CM}}$. For $m\geq 0$, we write $x\sim_m y$ if all of the following conditions hold:
\begin{itemize}
\item[(i)] The points $x$ and $y$ lie in the same geometrically connected component of $M$.
\item[(ii)] If $m=0$, then $x$ and $y$ have the same lattice chains (see~\ref{latticechains}).
\item[(iii)] If $m>0$, then there exists $t\in\mathcal{K}_{x,m}^1$ with $y=x^t$.
\end{itemize}
\end{defn}



We gather the following facts:
\begin{enumerate}
\item If $x\sim_m y$, then $x\sim_{m'} y$ for any $m'<m$.
\item If $x\sim_m y$ for all $m$, then $x=y$.
\item If $x\sim_m y$, then $\mathcal{K}_{x,m}=\mathcal{K}_{y,m}$.
\item The relation $\sim_m$ is preserved by the action of $G\times B^\times\times W_F$.
\item If $x\in M^E$, let $[x]_m$ be its equivalence class under $\sim_m$.  Then the stabilizer in $G\times B^\times$ of $[x]_m$ is exactly $\mathcal{K}_{x,m}^1$.
\item If $x\in M^{E}$ and $y\in M^{E'}$ have $x\sim_m y$, then $E=E'$ {\em except} in the following scenario:  $m=0$, $E$ and $E'$ are the two ramified extensions of $F$, $x$ and $y$ lie in the same connected component of $M$, and $x$ and $y$ have the same lattice chains.
\end{enumerate}

The curve $\mathfrak{X}_{x,m}$ does not depend on the choice of $x$ within the similarity class $[x]_m$ in the following strong sense.

\begin{prop}\label{independenceofx}  Let $x\sim_m y$.  Suppose there exists $t\in \mathcal{K}_{x,m}^1$ with $y=x^t$.
There is a unique $\mathcal{K}_{x,m}^1$-equivariant isomorphism $\phi\from\mathfrak{X}_{x,m}\to\mathfrak{X}_{y,m}$ for which $\phi(P_{x,m}^t)=P_{y,m}$.  The isomorphism $\phi$ does not depend on $t$.
\end{prop}

\begin{proof}  Certainly there exists a $\mathcal{K}_{x,m}^1$-equivariant isomorphism $\phi\from\mathfrak{X}_{x,m}\to\mathfrak{X}_{y,m}$, because the equations defining these curves and the actions of the group $\mathcal{K}_{x,m}^1$ on them are identical.  Then $\phi(P_{x,m})^t$ satisfies the properties of Prop.~\ref{basepoints}, and therefore there exists a $\mathcal{K}_v^1$-equivariant automorphism $f$ of $\mathfrak{X}_{y,m}$ carring $P_{y,m}$ onto $\phi\left(P_{x,m}^{t^{-1}}\right)$.  Renaming $\phi$ as $f^{-1}\circ \phi$, we have a $\mathcal{K}_{v}^1$-equivariant isomorphism $\phi\from\mathfrak{X}_{x,m}\to\mathfrak{X}_{y,m}$ carrying $P_{x,m}^t$ onto $P_{y,m}$.  Again by Prop.~\ref{basepoints}, this isomorphism is {\em unique}.
\end{proof}

We extend Prop.~\ref{independenceofx} to the case where $x\in M^{E_1}$, $y\in M^{E_2}$, and $x\sim_0 y$.  Choose uniformizers $\varpi_{E_1}$ and $\varpi_{E_2}$ for $E_1$ and $E_2$.  Write $j_y(\varpi_{E_2})=j_x(\varpi_{E_1})t$; then $t\in\mathcal{K}_{x,0}^1$.  Referring to~\ref{ramifiedlevelzero}, the action of $t$ on $\mathfrak{X}_{x,0}$ is an automorphism of the form $X\mapsto \alpha X$ for some $\alpha\in k_B^\times$.  Let $u\in \overline{k}$ be a root of $u^2=\alpha$;  then $X\mapsto uX$ is a $\mathcal{K}_{x,0}^1$-equivariant isomorphism from $\mathfrak{X}_{x,0}$ onto $\mathfrak{X}_{y,0}$.

Let $S_m$ be the set of similarity classes under $\sim_m$, and let $$S=\coprod_{m\geq 0} S_m$$ be the disjoint union of the $S_m$.  There is an obvious ``level" map $S\to \Z_{\geq 0}$ sending $S_m$ to $m$;  a CM point $x\in E$ therefore determines a section of this map sending $m$ to the similarity class $[x]_m$ under $\sim_m$ containing $m$.  The group $\GL_2(F)\times B^\times$ acts on $S$ in the obvious manner.  The set $S$ will serve as set of vertices in the dual graph in our construction of the stable Lubin-Tate curve $\mathfrak{X}$.

If $v\in S$, say $v=[x]_m$ for $x\in M^E$, then let $\mathcal{K}_v^1=\mathcal{K}_{x,m}^1$ and $\mathfrak{X}_v=\mathfrak{X}_{x,m}$.  In light of Prop.~\ref{independenceofx} and the paragraph that follows it, $v$ determines the $\mathcal{K}_v^1$-equivariant curve $\mathfrak{X}_v$ up to unique isomorphism.  Let $\tilde{\mathfrak{X}}$ be the disjoint union of the curves $\mathfrak{X}_v$ for $v\in S$.  Then $\tilde{\mathfrak{X}}$ admits an action of $\GL_2(F)\times B^\times$ in the following manner.   Let $t\in\GL_2(F)\times B^\times$ and let $\phi_t\from\mathfrak{X}_{x,m}\to\mathfrak{X}_{x^t,m}$ be the unique isomorphism satisfying $\phi_t\circ u=(t^{-1}ut)\circ \phi_t$ for $u\in\mathcal{K}_{x,m}^1$ and also $\phi(P_{x,m})^t=P_{x^t,m}$.  Then $\phi_t$ determines an isomorphism $\mathfrak{X}_v\to\mathfrak{X}_{v^t}$ which does not depend on the choice of $x$.




The curve $\tilde{\mathfrak{X}}$ also admits a semilinear action of $W_F$ which commutes with the action of $\GL_2(F)$.  Given $w\in W_F$ and $v=[x]_m\in S$,  let $\tilde{w}\in\mathcal{N}_{x}$ be a lift of $j_{x,W}(w)\in j_x(E^\times)\backslash\mathcal{N}_{x}$.   There is the automorphism $\Delta_{x,m}(\tilde{w})$ of $\mathfrak{X}_{x,m}$ as in Theorem~\ref{Weilaction}.   The element $w$ shall carry $\mathfrak{X}_v$ onto $\mathfrak{X}_{v^w}$ via the map $\phi_{\tilde{w}}\circ \Delta_{x,m}$.

\begin{prop}
\label{tildeX}
 Let $\pi$ be a smooth irreducible representation of $\GL_2(F)$.
Then $$\Hom_{\GL_2(F)}\left(\pi,H^1(\tilde{\mathfrak{X}},\QQ_\ell)\right)$$ is isomorphic to $\JL(\pi)\otimes \mathcal{L}_\ell(\check{\pi})$ if $\pi$ is supercuspidal, and is 0 otherwise.
\end{prop}

\begin{proof}  Let $R$ be a set of representatives for the quotient $S/(\GL_2(F)\times B^\times)$.  For each $v\in R$ we have that the stabilizer of $v$ in $\GL_2(F)\times B^\times$ is exactly $\mathcal{K}_{v}^1$.  Then as $\GL_2(F)\times B^\times\times W_F$-modules we have $$H^1(\tilde{\mathfrak{X}},\QQ_\ell)=\bigoplus_{v\in R}\Ind_{\mathcal{K}_v}^{\GL_2(F)\times B^\times} H^1(\mathfrak{X}_v,\QQ_\ell).$$  The proposition now follows from
Cor.~\ref{realizesLLC} once we observe the following:  (1) Every supercuspidal representation of $\GL_2(F)$ appears in $\Ind_{\mathcal{K}_v}^{\GL_2(F)\times B^\times} H^1(\mathfrak{X}_v,\QQ_\ell)$ for some $v\in R$, and (2) no supercuspidal representation appears in two summands of the above direct sum.

For (1), we note that the residue characteristic of $F$ is odd, so that every supercuspidal representation of $\GL_2(F)$ is of the form $\pi_\chi$ for some admissible pair $(E/F,\chi)$.  For (2), suppose $v=(x,m)$ and $v'=(x',m')$ contribute nontrivially to the above direct sum.  Write $x\in E$, $x'\in E'$.  If it happens that a supercuspidal representation $\pi$ appears in both summands, then we must have $\pi\isom\pi_\chi\isom\pi_{\chi'}$ for a admissible pairs $(E/F,\chi)$, $(E'/F,\chi')$.  But since both extensions are tamely ramified, this can only happen if $E=E'$ and $\chi$ and $\chi'$ have the same essential level.  By Thm.~\ref{CMpointbasic}, this implies that $v$ and $v'$ lie in the same orbit under $\GL_2(F)\times B^\times$.
\end{proof}

\subsection{The adjacency relation}  We now construct a graph whose vertex set is $S$.

If $E$ is one of the quadratic extensions of $F$, recall that there is a unique formal group $\Fc_E$ with endomorphisms by $\OO_E$.  For a point $x\in M^E$ represented by a triple $(\Fc_E,\iota,\alpha)$, we get a lattice chain $\Lambda_n=\alpha^{-1}(\gp_E^nT(\Fc_E))\subset F^2$.  Up to re-indexing, the lattice chain $\set{\Lambda_n}$ only depends on the isomorphism class of $x$.  We have $[\Lambda_n:\Lambda_{n+1}]=\#k_E$.

If $x,x'\in M$ are two CM points with respective lattice chains $\Lambda_n,\Lambda_n'$ then it might happen that there is a strict containment $\set{\Lambda_n}\subsetneq\set{\Lambda_n'}$.  For this to occur, we would need $x\in M^{E_0}$ and $y\in M^{E_i}$ for $i\in\set{1,2}$.  If this is the case we will say that the lattice chains for $x$ and $y$ {\em interlace}.

We now define a graph $\Gamma$ whose vertex set is $S$, with the following edges:

\begin{itemize}
\item[(i)] Draw an edge between the vertices $[x]_0$ and $[y]_0$ whenever $x$ and $y$ lie in the same connected component of $M$ and the lattice chains of $x$ and $y$ interlace.
\item[(ii)] For every CM point $x\in M$, and every $m\geq 0$, draw an edge between $[x]_{m+1}$ and $[x]_m$.
\end{itemize}

Call a vertex $v$ of $\Gamma$ unramified if it is of the form $[x]_m$ for $x\in M^{E_0}$, and ramified otherwise.
The graph $\Gamma_0$ is naturally isomorphic to the barycentric subdivision of the Bruhat-Tits tree $\mathcal{T}$ for $\GL_2(F)$.  Under this isomorphism, the unramified vertices of $\Gamma_0$ are in bijection with the vertices of $\mathcal{T}$, while the ramified vertices are in bijection with the midpoints of the edges of $\mathcal{T}$ in its barycentric subdivision.

\begin{figure}[!htb]
\centering
\includegraphics[scale=.7]{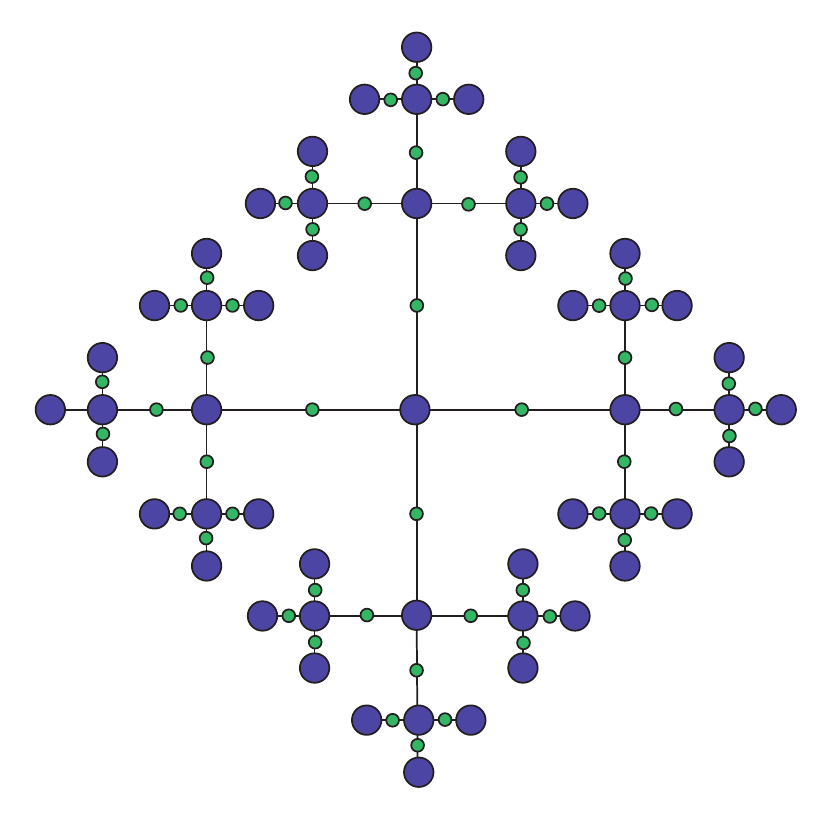}
\caption{A connected component of the graph $\Gamma_0$.  The case of $q=3$ is shown.  The larger blue vertices represent the projective plane curve $XY^q-X^qY=Z^{q+1}$, while the smaller green vertices represent the projective line.}
\label{BTtree}
\end{figure}

Let $E/F$ be a tame quadratic extension and let $v\in S_0$.  Let $\Gamma_v\subset \Gamma$ be the subgraph induced by the set of vertices of the form $[x]_m$, where $[x]_0=v$.  Then $\Gamma_v$ is a tree.

\begin{figure}[!htb]
\centering
\includegraphics[scale=.7]{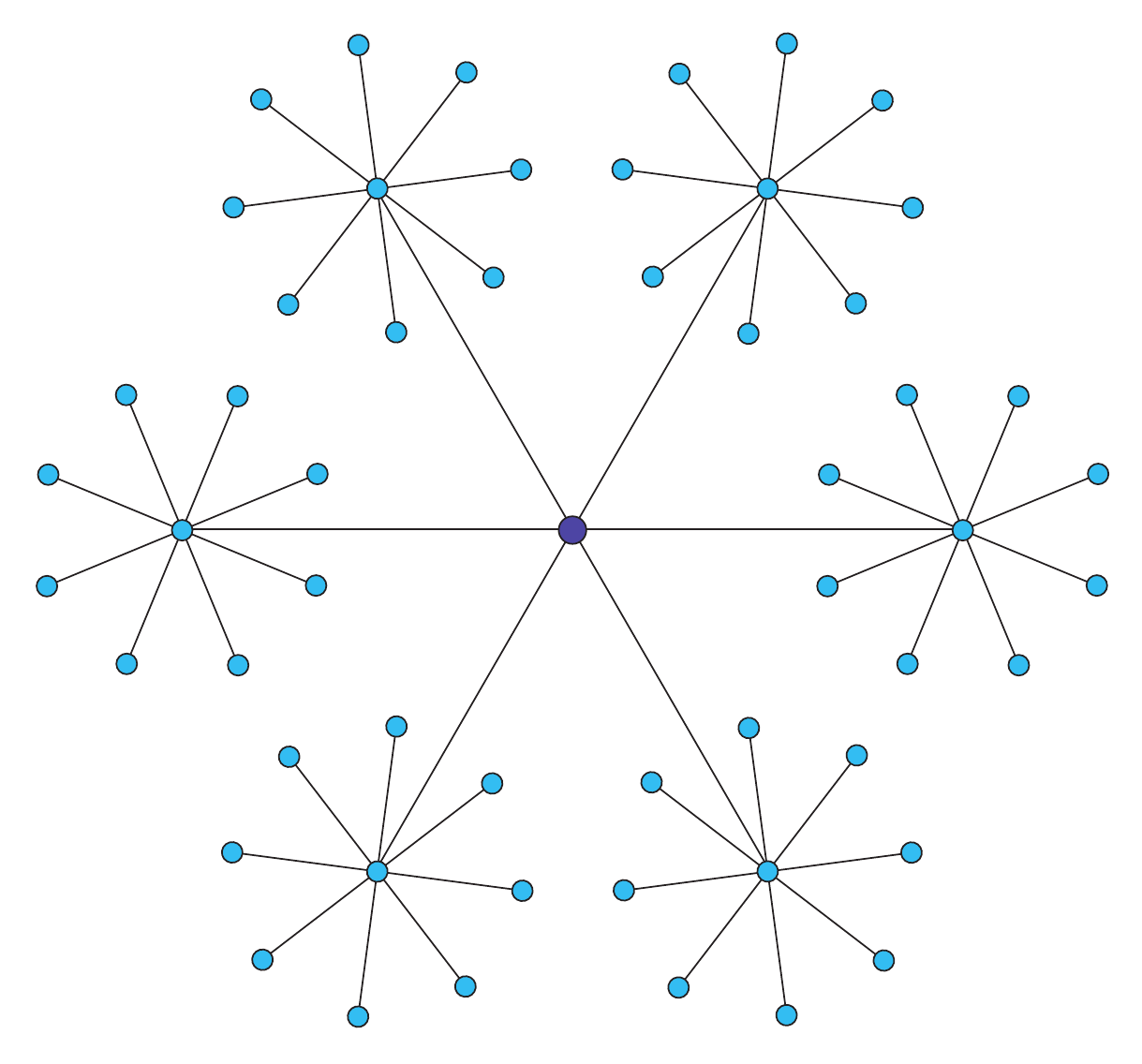}
\caption{The graph $\Gamma_v$ for an unramified vertex $v$.  The case $q=2$ is shown.  The central blue vertex is $v$ itself.    The vertices adjacent to $v$ are a torsor for the action of $\SL_2(\mathbf{F}_q)$.  The cyan vertices represent the projective plane curve $XY^q+X^qY=Z^{q+1}$.  The vertices of higher level adjacent to a particular cyan vertex are a torsor for the action of a nonabelian group of order $q^3$.  In the graph $\Gamma$, a copy of the above graph is glued to the graph $\Gamma_0$ in Fig.~\ref{BTtree} along each of the blue vertices.  }
\end{figure}

\begin{figure}[!htb]
\centering
\includegraphics[scale=.5]{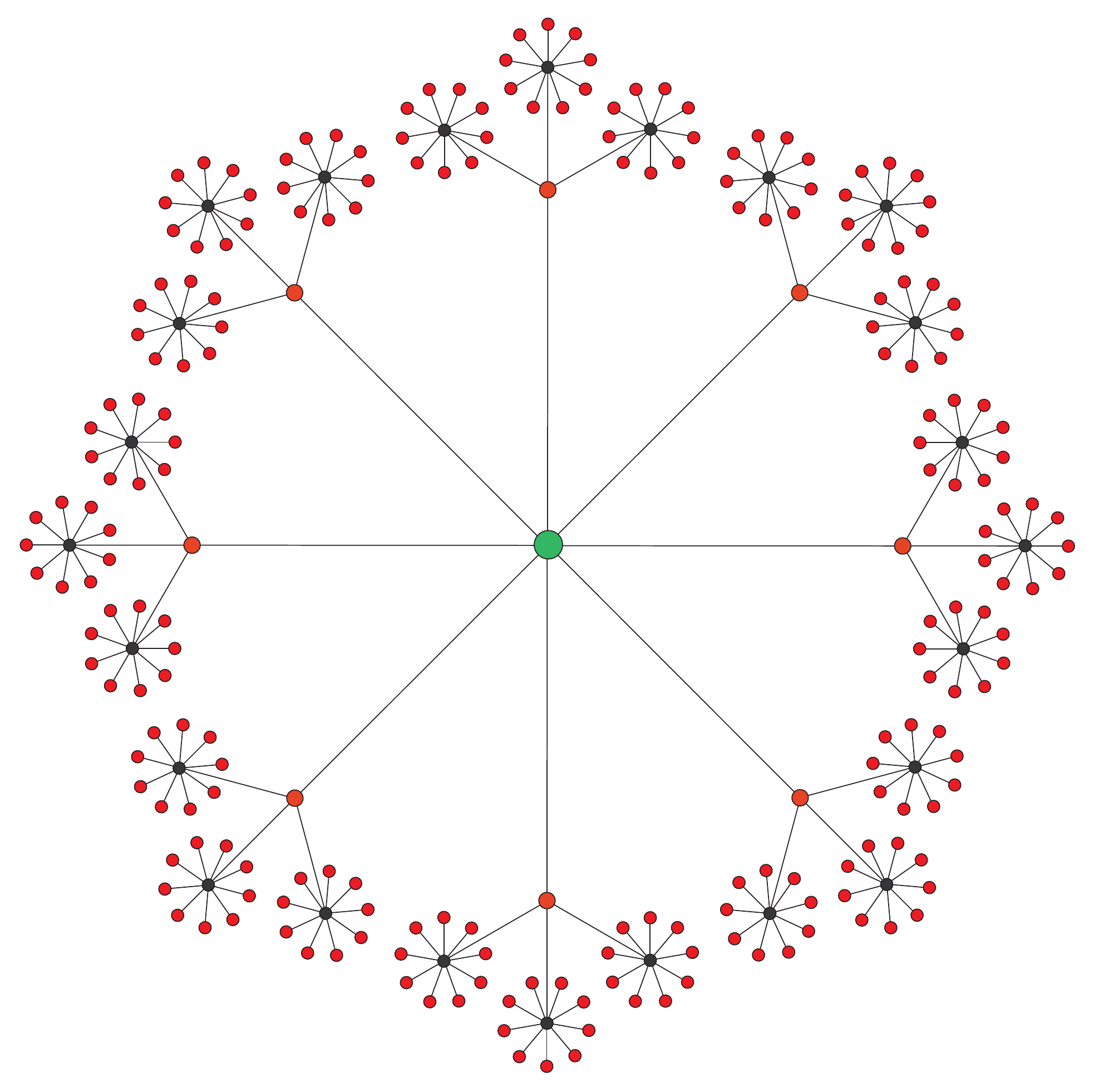}
\caption{The graph $\Gamma_v$ for a ramified vertex $v$.  The case $q=3$ is shown.  The central green vertex is $v$ itself.  The vertices adjacent to $v$ are a torsor for the action of $\mathbf{F}_{q^2}^\times$.  The red vertices represent the nonsingular hyperelliptic curve with affine equation $X^q-X=Y^2$;   the vertices of higher level adjacent to any particular cyan vertex are a torsor for the action of $\mathbf{F}_q$.  The black vertices represent the projective line;  the vertices of higher level adjacent to a particular black vertex are a torsor for the action of $\mathbf{F}_{q^2}$.  In the graph $\Gamma$, there are {\em two} copies of the above graph glued to each cyan vertex in $\Gamma_0$, corresponding to the two distinct unramified extensions $E/F$.}
\end{figure}


There are two types of ends of $\Gamma$:  those which stabilize in $\Gamma_0$, and those which pass through $S_m$ for every $m$.  The ``level 0" ends are in bijection with $\mathbf{P}^1(F)\times V(\mathcal{Z})$, and the
``unbounded" ends are in bijection with $M^{CM}$.


\subsection{Adjacency Data}  To complete the construction of $\mathfrak{X}$, we need to give, for each edge $e$ of $\Gamma$ joining the vertices $v,v'\in S$, a pair of points $P_{e,v}\in \mathfrak{X}_v$ and $P_{e,v'}\in \mathfrak{X}_{v'}$ which will be identified with each other in $\mathfrak{X}$.  Since the curve $\mathfrak{X}$ is expected to have an action of $\GL_2(F)\times B^\times$, this assignment must be consistent:  for every $t\in\GL_2(F)\times B^\times$, we must have $P_{e,v}^t=P_{e^t,v^t}$ and similarly for $v'$.

We begin with the case when the edge $e$ joins vertices $v,v'\in S$ both of level 0.  Without loss of generality we may assume that $x\in M^{E_0}$ and $x'\in M^{E_1}$ are such that $v=[x]_0$ and $v'=[x']_1$.  Then, via the choices of $x$ and $x'$, the curves $X_v$ and $X_{v'}$ may be identified with the curves $X_{x,0}$ and $X_{x',0}$:  These are the curves $x^qy-xy^q=1$ and the projective line, respectively.  Note that the points of $X_v$ at infinity are in correspondence with $\pr^1(k)$.

Since $v$ and $v'$ are adjacent, their respective lattice chains $\Lambda_n$ and $\Lambda_n'$ interlace:  $\set{\Lambda_n}\subset\set{\Lambda_n'}$ is ``every other lattice".  Let $I\subset \mathfrak{A}_x^\times$
 be the Iwahori subgroup which stabilizes the lattice chain $\set{\Lambda_n'}$.  Then $I$ fixes a unique point at infinity in $X_v$;  call this $P_{e,v}$.   If $\set{\Lambda_n}$ equals $\set{\Lambda_{2n}'}$, let $P_{e,v'}$ be the point $\infty\in X_{x',0}=\pr^1$.  If $\set{\Lambda_n}=\set{\Lambda_{2n+1}'}$, let $P_{e,v'}$ be the point $0\in X_{x',0}$.

Now suppose $m\geq 0$ and $v\in S_m$ and $v'\in S_{m+1}$ are adjacent.  This means there exists $x\in M^{\text{CM}}$ such that $v=[x]_m$ and $v'=[x]_{m+1}$.  We let $P_{e,v}$ be the base point $P_{x,m}\in\mathfrak{X}_v$ and $P_{e,v'}$ be the point at infinity in $\mathfrak{X}_{v'}$.

\subsection{Conclusion of Proof of Main Theorem}

Let $\mathfrak{X}$ be the curve obtained by applying the gluing the points $P_{e,v},P_{e,v'}\in\tilde{\mathfrak{X}}(\overline{k})$ for each edge $e$ of $\Gamma$ which joins the points $v,v'$.

\begin{prop}  $\mathfrak{X}$ is a stable curve.
\end{prop}

\begin{proof}  Since $\mathfrak{X}$ was obtained from the nonsingular curve $\tilde{\mathfrak{X}}$ by gluing points, we must check that no three components of $\mathfrak{X}$ meet at a single point.   Suppose that $x\sim_m y\in M^{\text{CM}}$, $v=[x]_m$, $w=[x]_{m+1}$, and $w'=[y]_{m+1}$, and that the components $\mathfrak{X}_v$, $\mathfrak{X}_{w}$ and $\mathfrak{X}_{w'}$ intersect at a common point.  This implies that $P_{x,m}=P_{y,m}\in \mathfrak{X}_v$.  Suppose there exists   $t\in\mathcal{K}_{x,m}^1$ be such that $y=x^t$, so that $P_{x,m}^t=P_{y,m}=P_{x,m}$.  But then $t$ lies in the stabilizer of $P_{x,m}$, which is $\mathcal{K}_{x,m+1}^1$.  Thus $x\sim_{m+1} y$ and $w=w'$.

If $y$ is not a translate of $x$ by an element of $\mathfrak{X}$, then $m=0$ and $x$ and $y$ are have CM by distinct ramified extensions of $F$.  Let $t\in\mathcal{K}_{x,0}^1$ be as in the remark following the proof of Prop.~\ref{independenceofx}, so that $P_{x,m}^t=P_{y,m}$ in $\mathfrak{X}_v$.  The proof now continues as in the previous paragraph.
\end{proof}

\begin{prop}  Then $$\Hom_{\GL_2(F)}\left(\pi,H^1(\mathfrak{X},\QQ_\ell)\right)$$ is isomorphic to $\JL(\pi)\otimes \mathcal{L}_\ell(\check{\pi})$ if $\pi$ is supercuspidal, and is 0 otherwise.
\end{prop}

\begin{proof}
Generally speaking, if $\mathfrak{X}$ is a stable curve with dual graph $\Gamma$, and $\tilde{\mathfrak{X}}$ is the normalization of $\mathfrak{X}$, then there is an exact sequence of commutative group schemes
\begin{equation}
\label{H1Gamma}
1\to H^1(\Gamma,\Z)\otimes\mathbf{G}_m\to \Jac\mathfrak{X}\to\Jac\tilde{\mathfrak{X}}\to 1
\end{equation}
Since $\Gamma$ is a disjoint union of trees, $H^1(\Gamma,\Z)=0$.  We conclude that $H^1(\mathfrak{X},\QQ_\ell)=H^1(\tilde{\mathfrak{X}},\QQ_\ell)$.  The proof now follows from Prop.~\ref{tildeX}.
\end{proof}



\section{Appendix: Some interesting curves}
\label{curves}

\subsection{On the hyperelliptic curve $Y^2=X^q-X$} Let $\mathfrak{X}_0$ be the nonsingular affine curve over $\overline{\F}_q$ with equation $y^2=x^q-x$.   The projective version of this equation, $Y^2Z^{q-2}=X^q-XZ^{q-1}$, has a singularity at $[0:1:0]$.  The singularity is resolved by means of the map of $\overline{\F}_q$-algebras
\begin{eqnarray*}
\frac{\overline{\F}_q[X,Z]}{Z^{q-2}-X^q+XZ^{q-1}}&\to& \frac{\overline{\F}_q[U,V]}{U^2-U^2V^{q-1}-V}\\
X&\mapsto& UV^{(q-3)/2}\\
Z&\mapsto& UV^{(q-1)/2}
\end{eqnarray*}
Let $\mathfrak{X}_\infty$ be the spectrum of the ring on the right;  then $\mathfrak{X}_\infty$ is a nonsingular affine curve.  Let $\mathfrak{X}$ be the proper curve obtained by gluing $\mathfrak{X}_0$ and $\mathfrak{X}_\infty$ along the morphism defined above.  Let $\infty$ be the unique closed point of $\mathfrak{X}\backslash\mathfrak{X}_0$;  this is the point $(u,v)=(0,0)$ on $\mathfrak{X}_\infty$.

The curve $\mathfrak{X}$ admits an action of the additive group $\F_q$ given by $[\alpha](x,y)= (x+\alpha,y)$ on $\mathfrak{X}_0$.   This group has $\infty\in\mathfrak{X}(\F_q)$ as its unique fixed point.  A calculation shows the multiplicity of $\infty$ as a fixed point for any $\alpha\in\F_q^\times$ is 3.





Let $\ell\nmid q$ be prime.  By the Lefshetz fixed-point theorem, the trace of $\alpha\in\F_q^\times$ acting on $H^*(\mathfrak{X},\Q_\ell)$ is $3$.  Since the traces of $\alpha$ on $H^0$ and $H^2$ are both 1, the trace of $\alpha$ on $H^1(\overline{\mathfrak{X}},\Q_\ell)$ is $-1$.  The genus of $\mathfrak{X}$ is $(q-1)/2$, so that $\dim H^1(\overline{\mathfrak{X}},\Q_\ell)=q-1$.
Therefore we have:
\begin{prop}\label{hyperellipticdecomp}
 As a module for $\F_q$, we have $$H^1(\mathfrak{X},\QQ_\ell)=\bigoplus_{\psi\neq 1}\psi,$$ the sum running over nontrivial characters $\psi\from\F_q\to\QQ_\ell^*$.
\end{prop}

Write $H^1$ for $H^1(\mathfrak{X},\QQ_\ell)$.  If $\psi$ is a nontrivial character of $\F_q$,  write $H^1[\psi]$ for the 1-dimensional $\psi$-eigenspace.  Since the action of $\F_q$ on $\mathfrak{X}$ is defined over $\F_q$, the geometric Frobenius element $\Fr\in\Aut\mathfrak{X}$ stabilizes each $H^1[\psi]$.

\begin{prop}  \label{frobeigenvalue} The eigenvalue of $\Fr$ on $H^1[\psi]$ is $-g_\psi$, where $g_\psi=\sum_{\alpha\in\F_q^{\times}}\psi(\alpha)\left(\frac{\alpha}{q}\right)$ is the quadratic Gauss sum.
\end{prop}

\begin{proof}
Consider the automorphism $\alpha\circ \Fr$ of $\mathfrak{X}$ which sends $(x,y)$ to $(x^q+a,y^q)$.    We count the fixed points of $\alpha\circ\Fr$.    If $(x,y)$ is an affine fixed point, then $x^q+a=x$ and $y^q=y$, so that $y^2=x^q-x=(x-a)-x=-a$ and also that $y\in\F_q$.    Naturally, there are always $q$ solutions to $x^q+a=x$.  It follows that there are $2q$ affine fixed points if $-a$ is a square in $\F_q^*$, $q$ affine fixed points if $a=0$, and none if $-a$ is a nonsquare.  In other words, the number of affine fixed points is $q\left(1+\left(\frac{-a}{q}\right)\right)$.  The point $\infty$ is always a fixed point for $\alpha\circ\Fr$ of multiplicity 1.  Therefore $$\tr\left(\alpha\circ\Fr\vert H^1\right)=-q\left(1+\left(\frac{-\alpha}{q}\right)\right)+1.$$

The group algebra $\QQ_\ell[\F_q]$ acts on $H^1$.   For a nontrivial character $\psi\from\F_q\to\QQ_\ell$, the idempotent element $$e_\psi=\frac{1}{q}\sum_{\alpha\in\F_q}\psi^{-1}(\alpha)[\alpha]\in\QQ_\ell[\F_q]$$ projects any $\QQ_\ell[\F_q]$-module $M$ onto its $\psi$-eigenspace $M[\psi]$.  The eigenvalue of $\Fr$ on $H^1[\psi]$ is therefore the trace of $e_\psi\circ\Fr$ on $H^1$.  This is

\begin{eqnarray*}
\tr\left(\Fr\vert H^1[\psi]\right) &=&\tr\left(e_\psi\circ\Fr\vert H^1[\psi]\right)\\
&=&\frac{1}{q}\sum_{\alpha\in\F_q}\psi^{-1}(\alpha)\tr\left(\alpha\circ\Fr\vert H^1\right)\\
&=&\frac{1}{q}\sum_{\alpha\in\F_q}\psi^{-1}(\alpha)\left[-q\left(1+\left(\frac{-\alpha}{q}\right)\right)+1\right]\\
&=&-\sum_{\alpha\in\F_q} \psi(\alpha)\left(\frac{\alpha}{q}\right),
\end{eqnarray*}

thus proving the proposition. \end{proof}

\subsection{On the Hermitian curve $Y^{q+1}=X^q+X$}\label{Hermitian}  Let $\mathfrak{X}$ be the nonsingular projective plane curve with affine equation $Y^{q+1}=X^q+X$.  This curve has genus $g=q(q-1)/2$.   The curve $\mathfrak{X}$ is known as the Hermitian curve.  It is a {\em maximal} curve over $\mathbf{F}_{q^2}$, meaning that
$$\#\mathfrak{X}(\mathbf{F}_{q^2}) = q^3 +1 = q^2 + 1 + 2gq$$ is the maximum number of $\mathbf{F}_{q^2}$-rational points for any nonsingular projective curve of genus $g$.

Let $\Phi\in\Aut\mathfrak{X}$ be the geometric Frobenius.  By the Lefshetz fixed-point theorem applied to $\Phi$, the maximality of $\mathfrak{X}$ is equivalent to:
\begin{prop}  \label{hermitian} $\Phi$ acts on $H^1(\mathfrak{X},\QQ_\ell)$ as the scalar $-q$.
\end{prop}

\begin{rmk}  $\mathfrak{X}$ is isomorphic over $\mathbf{F}_{q^2}$ to the Deligne-Lusztig curve $XY^q-X^qY=1$, so Prop.~\ref{hermitian} applies to that curve as well.
\end{rmk}

The automorphism group $\Aut_{\mathbf{F}_{q^2}}\mathfrak{X}$ is a unitary group in three variables.  The Borel subgroup of that unitary group is the group $Q\subset\PGL_3(\mathbf{F}_{q^2})$ of matrices of the form
$$
\left(\begin{matrix} \alpha & \beta & \gamma \\ & \alpha^q & \beta^q \\ & & \alpha \end{matrix}\right)
$$
which satisfy $\alpha\gamma^q+\alpha^q\gamma=\beta^{q+1}$.
Let $P\subset Q$ be the subgroup of matrices with $\alpha=1$ and $\beta=0$, so that $P$ is isomorphic to the group of elements of $\gamma\in \mathbf{F}_{q^2}$ with $\gamma^q+\gamma=0$.  Then $P$ lies in the center of $Q$.

Fix a character $\psi$ of $\mathbf{F}_q$ with values in $\QQ_\ell$.  For each $\alpha\in \mathbf{F}_{q^2}$, let $\psi_\alpha$ be the character $\gamma\mapsto \psi(\tr_{\mathbf{F}_{q^2}/\mathbf{F}_q} \alpha \gamma)$ of $P$.  Then $\psi_{\alpha}$ is nontrivial if and only if $\alpha\not\in\mathbf{F}_q$.  For each such $\alpha$, let $\tau_\alpha$ be the $\psi_{\alpha}$-isotypic component of the $Q$-module $H^1(\mathfrak{X},\QQ_\ell)$.

Let $D\subset Q$ be the subgroup of diagonal matrices.  The following is a straightforward application of the Lefshetz fixed-point theorem together with some representation theory.
\begin{prop}  \label{hermitiandecomp}
\mbox{}
\begin{itemize}
\item[(i)] $\tau_\alpha$ is an irreducible $Q$-module.
\item[(ii)] For every $\zeta\in D\backslash\set{1}$, we have $\tr\tau_\alpha(\zeta)=-1$.
\item[(iii)] $H^1(\mathfrak{X},\QQ_\ell)=\bigoplus_{\alpha} \tau_\alpha$, where $\alpha$ runs over a set of representatives for $\mathbf{F}_{q^2}\backslash\mathbf{F}_{q}$ modulo $\mathbf{F}_q^\times$.
\end{itemize}

\end{prop}

\bibliographystyle{amsalpha}
\bibliography{ENLTT}

\end{document}